\documentclass{amsart}
\usepackage{amsmath,amssymb}
\usepackage{amsrefs}
\usepackage{paralist}
\usepackage{mathrsfs}
\usepackage[all]{xy}

\date{\today}

\newtheorem{thm}{Theorem}[section]
\numberwithin{equation}{section}

\newtheorem{cor}[thm]{Corollary}
\newtheorem{lemma}[thm]{Lemma}
\newtheorem{prop}[thm]{Proposition}

\theoremstyle{definition}
\newtheorem{definition}[thm]{Definition}

\theoremstyle{remark}
\newtheorem{remark}[thm]{Remark}
\newtheorem{example}[thm]{Example}
\newtheorem{claim}{Claim}
\def\mathcs{C^{*}}
\newcommand{\cs}{\ensuremath{\mathcs}}

\DeclareMathSymbol{\rtimes}{\mathbin}{AMSb}{"6F}

\def\T{\mathbf{T}}
\def\Z{\mathbf{Z}}

\DeclareMathOperator*{\supp}{supp}
\def\set#1{\{\,#1\,\}}
\def\sset#1{\{#1\}}

\def\restr#1{_{{|#1}}}

\newbox\hidebox
\def\spechide#1{\setbox\hidebox=\hbox{$#1$}
\hbox to\wd\hidebox{$\box\hidebox^\wedge$\hss}}
\makeatletter
\def\labelenumi{\textnormal{(\@alph\c@enumi)}}
\def\theenumi{\@alph \c@enumi}
\def\labelenumii{\textnormal{(\@roman\c@enumii)}}
\def\theenumii{\@roman \c@enumii}
\newcount\charno
\def\alphapart#1{\charno=96
\advance\charno by#1\char\charno}

\makeatother
\def\<{\langle}
\def\>{\rangle}
\let\ipscriptstyle=\scriptscriptstyle
\def\lipsqueeze{{\mskip -3.0mu}}
\def\ripsqueeze{{\mskip -3.0mu}}
\def\ipcomma{\nobreak\mathrel{,}\nobreak}
\newbox\ipstrutbox
\setbox\ipstrutbox=\hbox{\vrule height8.5pt% depth 3.5pt
width 0pt}
\def\ipstrut{\copy\ipstrutbox}
\def\lip#1<#2,#3>{\mathopen{\relax_{\ipstrut\ipscriptstyle{
#1}}\lipsqueeze
\langle} #2\ipcomma #3 \rangle}
\def\blip#1<#2,#3>{\mathopen{\relax_{\ipstrut
\ipscriptstyle{ #1}}\lipsqueeze\bigl\langle} #2\ipcomma #3 \bigr\rangle}
\def\rip#1<#2,#3>{\langle #2\ipcomma #3
\rangle_{\ripsqueeze\ipstrut\ipscriptstyle{#1}}}
\def\brip#1<#2,#3>{\bigl\langle #2\ipcomma #3
\bigr\rangle_{\ripsqueeze\ipstrut\ipscriptstyle{#1}}}
\def\angsqueeze{\mskip -6mu}
\def\smangsqueeze{\mskip -3.7mu}
\def\trip#1<#2,#3>{\langle\smangsqueeze\langle #2\ipcomma #3
\rangle\smangsqueeze\rangle_{\ripsqueeze\ipstrut\ipscriptstyle{#1}}}
\def\btrip#1<#2,#3>{\bigl\langle\angsqueeze\bigl\langle #2\ipcomma
#3
\bigr\rangle
\angsqueeze\bigr\rangle_{\ripsqueeze\ipstrut\ipscriptstyle{#1}}}
\def\tlip#1<#2,#3>{\mathopen{\relax_{\ipstrut\ipscriptstyle{
#1}}\lipsqueeze \langle\smangsqueeze\langle} #2\ipcomma #3
\rangle\smangsqueeze\rangle}
\def\btlip#1<#2,#3>{\mathopen{\relax_{\ipstrut\ipscriptstyle{
#1}}\lipsqueeze
\bigl\langle\angsqueeze\bigl\langle} #2\ipcomma #3
\bigr\rangle\angsqueeze\bigr\rangle}

\def\ip(#1|#2){(#1\mid #2)}
\def\bip(#1|#2){\bigl(#1 \mid #2\bigr)}
\def\Bip(#1|#2){\Bigl( #1 \bigm| #2 \Bigr)}
\IfFileExists{mathrsfs.sty}{\usepackage{mathrsfs}}
{\let\mathscr\mathcal}

\newcommand\go{G^{(0)}}
\newcommand\ho{H^{(0)}}

\newcommand\lc{\beta}
\newcommand\lhlc{locally Hausdorff, locally compact}
\newcommand\ts{\tilde s}
\newcommand\trr{\tilde r}

\newcommand\xmg{G\backslash X}
\newcommand\ymg{G\backslash Y}
\newcommand\N{\mathbf {N}}
\IfFileExists{mathrsfs.sty}{\usepackage{mathrsfs}}
{\let\mathscr\mathcal} 
\newcommand\cc{\mathscr{C}} %% This can be redefined to
\newcommand\ndot{\cdot}
\newcommand\F{\mathscr{F}}
                            
\expandafter\ifx\csname BibSpec\endcsname\relax\else
\BibSpec{collection.article}{%
    +{}  {\PrintAuthors}                {author}
    +{,} { \textit}                     {title}
    +{.} { }                            {part}
    +{:} { \textit}                     {subtitle}
    +{,} { \PrintContributions}         {contribution}
    +{,} { \PrintConference}            {conference}
    +{}  {\PrintBook}                   {book}
    +{,} { }                            {booktitle}
    +{,} { }                            {series}
    +{,} { \voltext}                    {volume}
    +{,} { }                            {publisher}
    +{,} { }                            {organization}
    +{,} { }                            {address}
    +{,} { \PrintDateB}                 {date}
    +{,} { pp.~}                        {pages}
    +{,} { }                            {status}
    +{,} { \PrintDOI}                   {doi}
    +{,} { available at \eprint}        {eprint}
    +{}  { \parenthesize}               {language}
    +{}  { \PrintTranslation}           {translation}
    +{;} { \PrintReprint}               {reprint}
    +{.} { }                            {note}
    +{.} {}                             {transition}
}
\BibSpec{article}{%
    +{}  {\PrintAuthors}                {author}
    +{,} { \textit}                     {title}
    +{.} { }                            {part}
    +{:} { \textit}                     {subtitle}
    +{,} { \PrintContributions}         {contribution}
    +{.} { \PrintPartials}              {partial}
    +{,} { }                            {journal}
    +{}  { \textbf}                     {volume}
    +{}  { \PrintDatePV}                {date}
    +{,} { \eprintpages}                {pages}
    +{,} { }                            {status}
    +{,} { \PrintDOI}                   {doi}
    +{,} { available at \eprint}        {eprint}
    +{}  { \parenthesize}               {language}
    +{}  { \PrintTranslation}           {translation}
    +{;} { \PrintReprint}               {reprint}
    +{.} { }                            {note}
    +{.} {}                             {transition}
}
\BibSpec{book}{%
    +{}  {\PrintPrimary}                {transition}
    +{,} { \textit}                     {title}
    +{.} { }                            {part}
    +{:} { \textit}                     {subtitle}
    +{,} { \PrintEdition}               {edition}
    +{}  { \PrintEditorsB}              {editor}
    +{,} { \PrintTranslatorsC}          {translator}
    +{,} { \PrintContributions}         {contribution}
    +{,} { }                            {series}
    +{,} { \voltext}                    {volume}
    +{,} { }                            {publisher}
    +{,} { }                            {organization}
    +{,} { }                            {address}
    +{,} { pp.~}                        {pages}
    +{,} { \PrintDateB}                 {date}
    +{,} { }                            {status}
    +{}  { \parenthesize}               {language}
    +{}  { \PrintTranslation}           {translation}
    +{;} { \PrintReprint}               {reprint}
    +{.} { }                            {note}
    +{.} {}                             {transition}
}
\fi

\usepackage[normalem]{ulem} % sout stuff1
\usepackage{color}
\definecolor{Dgreen}{cmyk}{0.93,0.33,0.92,0.25} %% Dartmouth Green!

\allowdisplaybreaks

\begin{document}

\title[Amenability of Groupoids]{Amenability of
  Groupoids Arising from Partial Semigroup Actions and Topological
  Higher Rank Graphs}

\author{Jean N. Renault}
\address{D\'epartment de Math\'ematiques\\ Universit\'e d'Orl\'eans et CNRS (UMR 7349 et FR 2964)\\ 45067
Orl\'eans Cedex 2, FRANCE}
\email{jean.renault@univ-orleans.fr}

\author{Dana P. Williams}
\address{Department of Mathematics \\ Dartmouth College \\ Hanover, NH
03755-3551 USA}

\email{dana.williams@Dartmouth.edu}

\date{January 15, 2015; Revised March 16, 2015}

\thanks{The second author was supported by a grant from the
  Simons Foundation.}  


\begin{abstract}
  We consider the amenability of groupoids $G$ equipped with a
  group valued cocycle $c:G\to Q$ with amenable kernel $c^{-1}(e)$.
  We prove a general result which implies, in particular, that $G$ is
  amenable whenever $Q$ is amenable and if there is countable set $D\subset
  G$ such that $c(G^{u})D=Q$ for all $u\in\go$.

  We show that our result is applicable to groupoids arising from
  partial semigroup actions.  We explore these actions in detail and
  show that these groupoids include those arising from directed
  graphs, higher rank graphs and even topological higher rank graphs.
  We believe our methods yield a nice alternative groupoid approach to
  these important constructions.
\end{abstract}
\maketitle

%\tableofcontents 

\section{Introduction}
\label{sec:introduction}

It is often important to establish the amenability of groupoids that
arise in applications.  For example, amenability implies the equality
of the reduced and universal norms on the associated groupoid
algebras.  This is important in the study of the Baum-Connes
conjecture for groupoid \cs-algebras as it is the reduced algebra that
plays the key role, while the universal algebra has the better
functorial properties.  For example, Tu has shown that the \cs-algebra
of an amenable groupoid with a Haar system satisfies the Baum-Connes
conjecture and the UCT \cite{tu:kt99}.  In the classification program,
amenability implies nuclearity which is typically a crucial
hypothesis.

The sort of groupoids we wish to focus on are those arising from the
much studied \cs-algebras associated to higher-rank graphs, and more
recently, to topological higher-rank graphs.  As a specific example,
we recently considered the \cs-algebras of groupoids associated to
$k$-graphs (see \cite{rswy:xx12}).  Such groupoids have a canonical
$\Z^{k}$-valued cocycle $c$, and it is not hard to see that
$c^{-1}(0)$ is amenable.  In some cases, $c$ is not only surjective,
but strongly surjective in that $c(G^{u})=\Z^{k}$ for all
$u\in\go$. Then the amenability of $G$ is a consequence of
\cite{anaren:amenable00}*{Theorem~5.3.14}.  However, there are
interesting examples in which $c$ need not be strongly surjective, and
examples show that the problem of the amenability of $G$ turns out to
be very subtle.  A result of Spielberg
\cite{spi:tams14}*{Proposition~9.3} asserts that if $G$ is \'etale and
if $c:G\to Q$ is a continuous cocycle into a countable amenable group,
then $G$ is amenable whenever $c^{-1}(e_{Q})$ is.  Although this
result is satisfactory for most $k$-graphs, the proof is unsatisfying
in that it circumvents groupoid theory by invoking the nontrivial
result that for \'etale groupoid, $\cs(G)$ is nuclear if and only if
$G$ is amenable \cite{qui:jams96}*{Corollary~2.17}.  In particular,
this result is valid only for \'etale groupoids.

Here we want to prove a general result that subsumes both cocycle
results from \cite{anaren:amenable00} and from \cite{spi:tams14}.
That such a result will have delicate hypotheses is foreshadowed by
the observation that one can have a surjective continuous cocycle $c$
from a groupoid into an amenable group $Q$ such that $c^{-1}(e_{G})$
amenable and still have $G$ fail to be amenable.  For example, let $Q$
be an amenable group with a nonamenable subgroup $N$ (which obviously
is not closed in $Q$).  Let $G$ be the group bundle $G=Q\coprod N$ and
let $c$ be the identity map: $c(\gamma)=\gamma$.

Nevertheless, we obtain a quite general result:
Theorem~\ref{thm-skewprod}.  It implies in particular, that if $c:G\to
Q$ is a continuous cocycle into an amenable group $Q$ with amenable
kernel such that there
is a countable set $D$ so that $c(G^{u})D=Q$ for all $u\in\go$, then
$G$ is amenable.  Even in this form, we recover both results above and
remove the hypothesis that $G$ be \'etale from Spielberg's result.

Although our results apply to topological groupoids with Haar systems,
it is interesting that our proof relies on the notions of Borel
groupoids, Borel amenability and Borel equivalence.  At a crucial
juncture, we are able to show that our groupoid is Borel equivalent to
a Borel amenable groupoid.  Since we also show that Borel equivalence
preserves Borel amenability, we can appeal to the result from
\cite{ren:xx13} which demonstrates that topological amenability is
equivalent to Borel amenabiltiy --- at least for locally compact
groupoids with Haar systems.

Having proved our cocycle result, it is crucial to show how it can be
applied to groupoids which are currently being studied in the
literature.  To do this, we show that our results can be applied to
groupoids arising from (partial) semigroup actions. A key r\^ole is
played by the notion of \emph{directed} action
(Definition~\ref{def-directed}) which gives a partition of the space
into orbits.  We explore these actions in detail.  Then, making
significant use of hard work due to Nica and Yeend, we are able to
show that such groupoids include those arising from directed graphs,
higher rank graphs and even topological higher rank graphs.  We think
our methods using semigroup actions yield a nice alternative groupoid
approach to these important constructions.

To prove our cocycle result, we work with \lhlc\ groupoids which are
always assumed to be second countable and to possess a Haar system.
Note that a second countable, \lhlc\ groupoid is the countable union
of compact Hausdorff sets.  Hence its underlying Borel structure is
standard.  In Section~\ref{sec:borel-amenability} we review the notion
of Borel amenability and some of the basic properties of Borel
groupoids we need in the sequel.  In
Section~\ref{sec:borel-equivalence} we recall the notion of Borel
equivalence and prove that it preserves Borel amenability.  In
Section~\ref{sec:cocycles}, we prove the main amenability result.  In
Section~\ref{sec:semigroup action} we introduce semigroup
actions. When the action is directed, there is a semi-direct product
groupoid. This condition puts into light two classes of semigroups:
Ore semigroups and quasi-lattice ordered semigroups. Under reasonable
hypotheses, our main result applies and the semi-direct product
groupoid is amenable.  In Section~\ref{sec:THRG}, we show how the
groupoid corresponding to a topological higher rank graph (and
therefore to the many subcases of topological higher rank graphs) can
be realized as the semi-direct product groupoid of a directed
semigroup action of $P=\N^{d}$. If fact, we consider $P$-graphs, where
$P$ is an arbitrary semigroup rather than simply $\N^d$ as in the original
definition. The natural assumption is that the semigroup be
quasi-lattice ordered.  If moreover the semigroup is a
subsemigroup of an amenable group and the graph satisfies a properness
condition, then this groupoid is amenable as well. Besides the
amenability problem, the section reveals a tight connection between
higher rank C*-algebras and Wiener-Hopf C*-algebras. In fact, we use
the techniques introduced by Nica in \cite{nic:jot92}; in particular
the Wiener-Hopf groupoid of a quasi-lattice ordered semigroup defined
by Nica appears as a particular case of our general construction for
topological higher rank graphs.

\subsubsection*{Acknowledgments}
\label{sec:acknowledgements}

We are very grateful to an anonymous referee for a thorough reading of
our paper and for pointing out the relevance of the work of Exel
\cite{exe:sf09} and Brownlowe-Sims-Vittadello
\cite{bsv:ijm13}.  In addition the referee pointed out a gap
in the original proof of Theorem~\ref{thm-app-directed-sg-actions},
and suggested a line of attack which resulted in an improved version
of the result.  

\section{Borel Amenability}
\label{sec:borel-amenability}

For the details on Borel groupoids, proper Borel amenability and
proper Borel $G$-spaces, we refer to \cite{anaren:amenable00} and to
\cite{anaren:amenable00}*{\S2.1a} in particular.  Recall that in order
for a groupoid to act on (the left) a 
  space $X$, we require a map $r_{X}:X\to\go$.  If there is no
  ambiguity, we write simply $r$ in place of $r_{X}$ and call it the
  \emph{moment map}.\footnote{The term projection map, anchor map and
  structure map have also been used in the literature.} As in
\cite{anaren:amenable00}, to avoid pathologies we will always assume
that our Borel spaces are analytic Borel spaces. We recall some of the
basics here.  If $X$ and $Y$ are Borel spaces and $\pi:X\to Y$ is
Borel surjection, then a \emph{$\pi$-system} is a family of measures
$m=\sset{m^{y}}_{y\in Y}$ such that each $m^{y}$ is supported in
$\pi^{-1}(y)$ and such that
\begin{equation*}
  y\mapsto \int_{X}f(x)\,d m^{y}(x)
\end{equation*}
is Borel for any nonnegative Borel function $f$ on $X$.  If $G$ is a
Borel groupoid, $X$ and $Y$ are Borel $G$-spaces and $\pi$ is
$G$-invariant, then we have that $m$ is \emph{invariant} if $\gamma
\cdot m^{y}= m^{\gamma\cdot y}$ whenever $s(\gamma)=r(y)$. (By
definition, $\gamma\cdot m^{y}(E)=m^{y}(\gamma^{-1}\cdot E)$.)
\begin{definition}[\cite{anaren:amenable00}*{Definition~2.1.1(b)}]
  \label{def-proper-amen}
  Suppose that $G$ is a Borel groupoid and that $X$ and $Y$ are Borel
  $G$-spaces.  A surjective Borel $G$-map $\pi:X\to Y$ is
  \emph{$G$-properly amenable} if there is a invariant Borel
  $\pi$-system $m=\sset{m^{y}}_{y\in Y}$ of probability measures on
  $X$.  Then we say that $m$ is an \emph{invariant mean for $\pi$}.
\end{definition}

\begin{remark}
  \label{rem-quotient-maps}
  Notice that if $\pi:X\to Y$ is a Borel $G$-map, then the induced map
  $\dot \pi:\xmg\to \ymg$ is Borel with respect to the quotient Borel
  structures.  Suppose that $\pi$ is $G$-properly amenable and that
  $\sset{\lambda^{y}}$ is an invariant mean for $\pi$. Let
  $p:X\to\xmg$ be the quotient map. Then the push forward
  $p_{*}\lambda^{y}$ is a probability measure supported on
 ~$\dot\pi^{-1}(\dot y)$.  By invariance, it depends only on $\dot y$
  and we can denote this measure by~$\dot\lambda^{\dot y}$.  If $b$ is
  a bounded Borel function on $\ymg$, then
  \begin{equation*}
    \int_{\ymg}b(\dot z)\,d\dot\lambda^{\dot y}(\dot z) = \int_{X}
    b(p(x))\,d\lambda^{y}(x). 
  \end{equation*}
  Hence $\sset{\dot\lambda^{\dot y}}_{\dot y\in \ymg}$ is a Borel
  $\dot\pi$-system of probability measures on $\xmg$.
\end{remark}

\begin{definition}[\cite{anaren:amenable00}*{Defintion~2.1.2}]
  \label{def-prop-amen}
  A Borel groupoid $G$ is \emph{proper} if the range map $r:G\to\go$ is
  $G$-properly amenable.  A Borel $G$-space $X$ is \emph{proper} if the
  projection $p:X*G\to X$ is $G$-properly amenable.
\end{definition}
\begin{remark}
  \label{rem-amen-gspace}
  Thus $X$ is a proper $G$-space if and only if there is a family
  $\sset{m^{x}}_{x\in X}$ of probability measures $m^{x}$ on $G$
  supported in $G^{r(x)}$ such that $x\mapsto m^{x}(f)$ is Borel for
  all non-negative Borel functions on $G$ and such that $\gamma\cdot
  m^{x}=m^{\gamma\cdot x}$.
\end{remark}

If $X$ and $Y$ are (left) $G$-spaces, then the fibered product $X*Y$
is a $G$-space with respect to the diagonal action.  If $X$ and $Y$
are Borel $G$-spaces, then we give $X*Y$ the Borel structure as a
subset of $X\times Y$ with Borel structure generated by the Borel
rectangles.  In particular, if $X$ is a $G$-space, then
$X*G=\set{(x,\gamma)\in X\times G:r(x)=r(\gamma)}$ is a $G$-space with
respect to the diagonal action.\footnote{In fact, $X*G$ is a groupoid
  --- sometimes called the transformation groupoid for $G$ acting on
  $X$.  This groupoid is denoted by $X\rtimes G$ in
  \cite{anaren:amenable00}. After identifying its unit space with $X$,
  the range map $(x,\gamma)\mapsto x$ is clearly $G$-equivariant.
  Hence, in Defintion~\ref{def-prop-amen}, we defined a $G$-space $X$
  to be properly amenable exactly when the groupoid $X\rtimes G$ is
  properly amenable.}
\begin{definition}
  \label{def-borel-amen}
  Suppose that $G$ is a Borel groupoid and that $X$ and $Y$ are Borel
  $G$-spaces.  A surjective Borel $G$-map $\pi:X\to Y$ is \emph{$G$-amenable}
  (or simply amenable if there is no ambiguity about $G$) if there is
  a sequence $\sset{m_{n}}_{n\in \N}$ of Borel systems of probability
  measures $y\mapsto m_{n}^{y}$ on $X$ which is approximately
  invariant in the sense that for all $\gamma\in G$, $\| \gamma\cdot
  m_{n}^{y} - m_{n}^{\gamma\cdot y}\|_{1}$ converges to $0$, where
  $\|\cdot\|_{1}$ is the total variation norm.  We say that
  $\sset{m_{n}}_{n\in\N}$ is an \emph{approximate invariant mean} for
  $\pi$.
\end{definition}

The notion of Borel amenability for groupoids was first formalized in
\cite{ren:xx13}*{Definition~2.1} --- however, here we use the
formulation in which conditions (ii)~and (iii) of
\cite{ren:xx13}*{Definition~2.1} have been replaced by (ii$'$).

\begin{definition}
  \label{def-borel-amen-grp-space}
  A Borel groupoid $G$ is (Borel) \emph{amenable} if the range map
  $r:G\to\go$ is $G$-amenable.  A Borel $G$-space $X$ is \emph{amenable} if
  the projection $p:X*G\to X$ is $G$-amenable.
\end{definition}

A key result about Borel amenable maps we need is the
following. 

\begin{lemma}
  \label{lem-key}
  Let $G$ be a Borel groupoid.  If there is a proper Borel $G$-space
  $X$ such that the moment map  $r:X\to\go$ is Borel amenable, then $G$
  is Borel amenable.
\end{lemma}
\begin{proof}
  Since, by assumption, the map $\pi_{X}:X*G\to X$ is properly
  amenable, there is an invariant system $m=\sset{m^{x}}_{x\in X}$ of
  probability measures for $\pi_{X}$.  We can view each $m^{x}$ as a
  measure on $G^{r_{X}(x)}$ such that $\gamma\cdot
  m^{x}=m^{\gamma\cdot x}$.  Since $r_{X}$ is amenable, there is an
  approximate invariant mean $\sset{\mu_{n}}$.  Then $\mu_{n}^{u}$
  is a probability measure on $r_{X}^{-1}(u)$, and for all $\gamma\in
  G$, $\lim_{n} \|\gamma\cdot \mu_{n}^{s(\gamma)}-
  \mu_{n}^{r(\gamma)}\|_{1}=0$.  Define
  \begin{equation*}
    \lambda_{n}^{u}= \int_{X} m^{x}\,d\mu_{n}^{u} (x).
  \end{equation*}
  Then $\lambda_{n}^{u}$ is a probability measure supported in
  $G^{u}$.  The system $\lambda_{n}=\sset{\lambda_{n}^{u}}_{u\in\go}$
  is certainly Borel, and
  \begin{align*}
    \|\gamma\cdot
    \lambda_{n}^{s(\gamma)}-\lambda_{n}^{r(\gamma)}\|_{1} &=
    \Bigl\|\int_{X}\gamma\cdot m^{y}\,d\mu_{n}^{s(\gamma)}(y) -
    \int_{X} m^{x} \,\mu_{n}^{r(\gamma)}(x) \Bigr\|
    \\
    &= \Bigl\| \int_{X}m^{\gamma\cdot x}\,d\mu_{n}^{s(\gamma)} (y) -
    \int_{X} m^{x}\, d\mu_{n}^{r(\gamma)}(x)\Bigr\| \\
    &= \Bigl\|\int_{X} m^{x}\,d(\gamma\cdot \mu_{n}^{s(\gamma)} -
    \mu_{n}^{r(\gamma)})(x) \Bigr\|\\
    &\le \|\gamma\cdot \mu_{n}^{s(\gamma)}-\mu_{n}^{r(\gamma)}\|_{1}.
  \end{align*}
  It follows that $\lambda$ is an approximate invariant mean for
  $r:G\to\go$.  Thus $G$ is Borel amenable.
\end{proof}

Of course, a proper Borel groupoid is Borel
  amenable.  Our next lemma should be compared with
  \cite{anaren:amenable00}*{Proposition~5.3.37}.

\begin{lemma}
  \label{lem-incr-union}
  Suppose that $G$ is a Borel groupoid which is the increasing union
  of a sequence of proper Borel subgroupoids $G_{n}$.  Then $G$ is
  Borel amenable.
\end{lemma}
\begin{proof}
  By assumption, we can find an system $m_{m}=(m_{n}^{u})_{u\in
      \go_{n}}$ of invariant probability measures on $G_{n}$; that is,
    $\gamma\cdot m_{n}^{s(\gamma) } = m_{n}^{r(\gamma)}$ for all
    $\gamma\in G_{n}$. Using the standard theory of disintegration of
    measures for example (see \cite{wil:crossed}*{Theorem~I.5}), we
    can extend each $m_{m}$ to a Borel system of 
    probability measures on all of $\go$.  Then 
  that $m=(m_{n})$ is an approximate invariant mean for $r:G\to\go$.
\end{proof}

\section{Borel Equivalence}
\label{sec:borel-equivalence}

The definition of equivalence for Borel groupoids is given in the
appendix of \cite{anaren:amenable00}.  In light of \cite{ren:xx13}, it
turns out to be a much more significant notion than originally
thought.

\begin{definition}[\cite{anaren:amenable00}*{Definition~A.1.11}]
  \label{def-borel-equi}
  Let $G$ and $H$ be Borel groupoids.  A \emph{$(G, H)$-Borel equivalence} is
  a Borel space $Z$ such that
  \begin{enumerate}
  \item $Z$ is a free and proper left Borel $G$-space,
  \item $Z$ is a free and proper right Borel $H$-space,
  \item The $G$- and $H$-actions commute,
  \item $r:Z\to\go$ induces a Borel isomorphism between $Z/H$ and
    $\go$, and
  \item $s:Z\to\ho$ induces a Borel isomorphism between $G\backslash
    Z$ and $\ho$.
  \end{enumerate}
  In this case, we say that $G$ and $H$ are \emph{Borel equivalent}.
\end{definition}

It is proved in \cite{anaren:amenable00}*{Theorem~2.2.17} that
equivalence of locally compact groupoids preserves (topological)
amenability.  Here we show a similar result holds in the Borel case
using virtually the same argument.
\begin{thm}
  \label{thm-borel-equiv}
  Suppose that $G$ and $H$ are equivalent Borel groupoids.  If $H$ is
  Borel amenable (resp., properly amenable), then so is $G$.
\end{thm}
\begin{proof}
  Let $Z$ be a $(G,H)$-equivalence.  Consider the commutative diagram
  \begin{equation*}
    \xymatrix@C+3pc{G\ar[r]^{r}&\go\\
      Z*G\ar[u]^{p_{G}}\ar[r]^{p_{Z}}\ar[d]_{p}&Z\ar[d]^{q}\ar[u]_{r}\\
      Z\ar[r]^{q}&G\backslash Z,}
  \end{equation*}
  where $p(z,\gamma):= \gamma^{-1}\cdot z$.

  First, assume that $G$ is properly amenable.  To see that $H$ is
  properly amenable it will suffice, by
  \cite{anaren:amenable00}*{Corollary~2.1.7}, to see that that $H$-map
  $s:Z\to\ho$ is properly amenable.
  % \footnote{This is the analogue of Lemma~\ref{lem-key} and
  % Corollary~\ref{cor-amen-g-spaces}.}

  Let $\lambda=\sset{\lambda^{u}}_{u\in\go}$ be an invariant mean for
  $r:G\to\go$. We first build an invariant mean $\lambda_{Z}$ for
  $p_{Z}:Z*G\to Z$: let $\lambda_{Z}^{z}$ be given by
  \begin{equation*}
    \lambda_{Z}^{z}(f)=\int_{G}f(z,\lambda)\,d\lambda^{r_{Z}(z)}(\gamma).
  \end{equation*}
  Consider the measure on $Z$ obtained by the push forward
  $q_{*}\lambda_{Z}^{z}$.  Since invariance implies 
  $\eta^{-1}\cdot \lambda^{r(\eta)}=\lambda^{s(\eta)}$ for all
  $\eta\in G$, it follows that
  \begin{align*}
    q_{*}\lambda_{Z}^{\eta\cdot z}(b)&=\int_{G}b(p(\eta\cdot
    z,\gamma))\,d\lambda_{Z}^{\eta\cdot z}(\gamma) \\
    &=\int_{G}b(\gamma^{-1}\eta\cdot z) \,d\lambda^{r_{Z}(\eta\cdot
      z)}(\gamma) = \int_{G}b(\gamma^{-1}\cdot z)\,d(\eta^{-1}\cdot
    \lambda^{r(\eta)})(\gamma)\\
    & = q_{*}\lambda_{Z}^{z}(b).
  \end{align*}
  Hence the push forward $q_{*}\lambda_{Z}^{z}$ depends only on $\dot
  z$ in $G\backslash Z$. Since $Z$ is a Borel equivalence, we can
  identify the quotient map $q:Z\to G\backslash Z$ with the moment map
  $s:Z\to\ho$.  Thus we get an $s$-system of measures
  $\nu=\sset{\nu^{v}}_{v\in\ho}$ where $\nu^{v}=q_{*}\lambda_{Z}^{z}$
  for any $z$ with $s(z)=v$.  It will suffice to see that $\nu$ is
  invariant.  But if $s(z)=r(h)$, then
  \begin{align*}
    \int_{Z}b(z)\,d(\nu^{r(h)}\cdot h)(z)
    &=\int_{G}b((\gamma^{-1}\cdot z)\cdot h)\,d\lambda^{r(z)}(\gamma)
    =\int_{G} b(\gamma^{-1}\cdot
    (z\cdot h))\,d\lambda^{r(z\cdot h}(\gamma)\\
    & = \int_{Z} b(z)\,d\nu^{s(h)}(z).
  \end{align*}
  This completes the proof for proper amenability.

  Now we assume that $G$ is Borel amenable and that
  $\lambda=\sset{\lambda_{n}}$ is an approximately invariant mean for
  $r:G\to\go$ with $\lambda_{n}=\sset{\lambda_{n}^{u}}_{u\in\go}$.  By
  Lemma~\ref{lem-key}, it will suffice to show that the $H$-map
  $s:Z\to\ho$ is Borel amenable.  The argument parallels that above
  argument for proper amenability.

  We begin by lifting each $\lambda_{n}$ to a system
  $\lambda_{n,Z}=\sset{\lambda_{n,Z}^{z}}_{z\in Z}$ of probability
  measures for $p_{Z}:Z*G\to Z$:
  \begin{equation*}
    \lambda_{n,Z}^{z}(f)=\int_{G}
    f(z,\gamma)\,d\lambda_{n}^{r(z)}(\gamma). 
  \end{equation*}
  We let $\mu_{n}^{z}=q_{*}\lambda_{n,Z}^{z}$.  Since
  $\lambda_{n,Z}$ is not necessarily invariant, we can't assert that
  $\mu_{n}^{z}$ depends only on $\dot z$.  However
  \begin{equation*}
    \int_{Z}b(w)\,d\mu_{n}^{z}(w)=\int_{G} b(\gamma^{-1}\cdot
    z)\,d\lambda_{n}^{r(z)}(\gamma), 
  \end{equation*}
  and $\mu_{n}^{z}$ is supported on $q^{-1}(z)=G\cdot z$.

  Since $Z$ is a proper $G$-space, by definition there is a
  $G$-invariant system of probability measures
  $\rho=\sset{\rho^{z}}_{z\in Z}$ for the Borel $G$-map $p_{Z}:Z*G\to
  Z$ which we view as a family of measures on $G$ such that $\supp
  \rho^{z}\subset G^{r(z)}$.  As in Remark~\ref{rem-quotient-maps},
  $\rho$ drops to a system $\dot \rho=\sset{\dot \rho^{\dot z}}_{\dot
    z\in G\backslash Z}$ for the quotient map $q:Z\to G\backslash Z$.
  Since $(z,\gamma)\mapsto \gamma^{-1}\cdot z$ identifies $G\backslash
  (Z*G)$ with $Z$, we can view these as measures on $Z$.  Explicitly,
  $\dot \rho^{\dot z}= p_{*}\rho^{z}$ and
  \begin{equation*}
    \dot\rho^{\dot z}(b)=\int_{G}b(\gamma^{-1}\cdot z)\,d\rho^{z}(\gamma).
  \end{equation*}
  Using the invariance of $\rho$ we get measures depending only on
  $\dot z$ supported on $q^{-1}(\dot z)$ by averaging the
  $\mu_{n}^{z}$ with respect to $\dot\rho^{\dot z}$:
  \begin{equation}
    \label{eq:4}
    \nu_{n}^{\dot z} = \int_{Z}\mu_{n}^{z}\,d\dot\rho^{\dot
      z}(z)=\int_{G} \mu_{n}^{\gamma^{-1}\cdot z}\,d\rho^{z}(\gamma).
  \end{equation}

  As in the first part of the proof, we use the fact that $Z$
  implements an equivalence to identify the quotient map $q:Z\to
  G\backslash Z$ with the moment map $s:Z\to\ho$.  Thus we get an
  $s$-system $\sset{\nu^{v}}_{v\in\ho}$ where $\nu^{v}=\nu^{\dot z}$
  for any $z$ such that $s(z)=v$.  We will complete the proof by
  showing that $\nu$ is an approximately invariant mean for $s$.
  Therefore we need to see that for all $h\in H$,
  $\|\nu^{r(h)}_{n}\cdot h - \nu_{n}^{s(h)}\|_{1}$ tends to zero with
  $n$.  Notice that if $(z,\gamma)\in Z*G$, then
  \begin{equation*}
    \|\mu_{n}^{z}-\mu_{n}^{\gamma^{-1}\cdot z}\|_{1}\le
    \|\lambda_{n}^{r(\gamma)}-\lambda\cdot \lambda_{n}^{s(\gamma)}\|_{1}.
  \end{equation*}
  Next we claim that for fixed $z\in Z$, $\lim_{n}\|\nu^{\dot z}
  -\mu_{n}^{z}\|_{1}=0$.  To see this we employ \eqref{eq:4}, view
  $\rho^{z}$ as a measure on $G^{r(z)}$, and deduce that
  \begin{equation}
    \label{eq:5}
    \|\nu_{n}^{\dot z}-\mu_{n}^{z}\|_{1}\le \int_{G}
    \|\mu_{n}^{\gamma^{-1}\cdot z} - \mu_{n}^{z}\|\,d\rho^{z}(\gamma)
    \le \int_{G} \|\lambda_{n}^{r(\gamma)}-\gamma\cdot
    \lambda_{n}^{s(\gamma)}\|_{1} \,d\rho^{z}(\gamma).
  \end{equation}
  Since for each $\gamma$,
  $\lim_{n}\|\lambda_{n}^{r(\gamma)}-\gamma\cdot
  \lambda_{n}^{s(\gamma)}\|_{1}=0$, we see that \eqref{eq:5} goes to
  zero by the Lebesgue Dominated Convergence Theorem.

  Now fix $h\in H$ and $\epsilon>0$.  Let $z\in Z$ be such that
  $s(z)=r(h)$ and let $z'=z\cdot h$ and observe that $\mu_{n}^{z}\cdot
  h= \mu_{n}^{z'}$.  Let $M$ be such that $n\ge M$ implies that
  \begin{equation*}
    \|\nu_{n}^{\dot z'}-\mu_{n}^{z'}\| <\frac\epsilon2.
  \end{equation*}
  Then for $n\ge M$, we have
  \begin{align*}
    \|\nu_{n}^{r(h)}\cdot h -\nu_{n}^{s(h)}\|_{1}&\le
    \|\nu_{n}^{r(h)}\cdot h -\mu_{n}^{z}\cdot h\|_{1}+\|
    \mu_{n}^{z}\cdot h - \mu_{n}^{\dot z}\| + \|\mu_{n}^{\dot z} -
    \nu_{n}^{s(h)}\| \\
    &= \|\nu_{n}^{\dot z'}-\mu_{n}^{z'}\|_{1}+0+\|\mu_{n}^{z'} -
    \nu_{n}^{\dot z'}\| <\epsilon.
  \end{align*}
  Thus $s:Z\to\ho$ is Borel amenable.  This completes the proof.
\end{proof}

% \begin{remark}
%   \label{rem-what-we-used}
%   In the proof of Theorem~\ref{thm-borel-equiv}, we don't invoke all
%   the properties of an equivalence.  For the proof, it suffices for
%   $G$ to be a Borel amenable group acting properly on an amenable
%   Borel $H$-space so that the moment map $s:Z\to\ho$ induces a Borel
%   isomorphism with $\ho$.  Then it follows that $H$ must be amenable.
%   (In particular, neither the $G$- nor $H$-action need be free.)
% \end{remark}

We close this section with two technical results which will be of use
in the next section.  Recall from 
\cite{ren:xx13}*{Definition~2.3} that a \emph{Borel approximate invariant
density} on $G$ is a sequence $(g_{n})$ of non-negative Borel functions
on $G$ such that
\begin{gather*}
  \int_{G} g_{n}(\gamma)\,d\lambda^{u}(\gamma)\le 1 \quad\text{for all
  $n$,}\quad \int_{G}g_{n}(\gamma)\,d\lambda^{u}(\gamma)\to
1\quad\text{for all $u\in\go$ and} \\
\int_{G}\bigl|
g_{n}(\gamma^{-1}\gamma')-g_{n}(\gamma')\bigr|\,d\lambda^{r(\gamma)}
(\gamma')\to 0\quad\text{for all $\gamma\in G$.}
\end{gather*}

\begin{lemma}
  \label{lem-borel-cover}
  Let $G$ be a Borel groupoid with a Borel Haar system
  $\sset{\lambda^{u}}_{u\in\go}$.  Let $\sset{Y_{i}}_{i=1}^{\infty}$
  be a countable cover of $\go$ by invariant Borel subsets such that
  each $G\restr {Y_{i}}$ is Borel amenable.  Then $G$ is Borel
  amenable.
\end{lemma}
\begin{proof}
  Since each $Y_{i}$ is invariant, $u\in Y_{i}$ implies that
  $(G\restr{Y_{i}})^{u} = G^{u}$.  Hence $\lambda$ restricts to a
  Borel Haar system on $G\restr{Y_{i}}$.  Hence by
  \cite{ren:xx13}*{Proposition~2.4}, the Borel amenability of
  $G\restr{Y_{i}}$ implies that there is a Borel approximate invariant
  density),
  $\sset{g_{n}^{i}}_{n=1}^{\infty}$, on $G\restr {Y_{i}}$.

  Let $B_{1}=Y_{1}$ and if $i\ge 2$, let $B_{i}=Y_{i
  }\setminus\bigcup_{j=1}^{i-1}Y_{j}$.  Then the $\sset{B_{i}}$ are a
  pairwise disjoint cover of $\go$ by invariant Borel sets (some of
  which might be empty).  Let $b^{i}$ be the characteristic function
  of $B_{i}$.  Then each $b^{i}$ is a Borel function on $\go$ (taking
  the values $0$ and $1$) such that $b_{i}$ vanishes off $Y_{i}$ and
  for each $u\in\go$, there is one and only one $i$ such that
  $b^{i}(u)=1$.  In particular, we trivially have
  \begin{equation*}
    \sum_{i=1}^{\infty} b^{i}(u)=1\quad\text{for all $u\in \go$.}
  \end{equation*}

  For each $n$, define $g_{n}$ on $G$ by
  \begin{equation*}
    g_{n}(\gamma)=\sum_{i} g_{n}^{i}\cdot b^{i},
  \end{equation*}
  where $g_{n}^{i}\cdot
  b^{i}(\gamma)=g_{n}^{i}(\gamma)b^{i}(s(\gamma))$.

  By invariance,
  \begin{equation*}
    \int_{G}g_{n}(\gamma)\,d\lambda^{u}(\gamma) =
    \int_{G}g_{n}^{i}(\gamma) \,d\lambda^{u}(\gamma),
  \end{equation*}
  where $i$ is such that $b^{i}(u)=1$.  Now it is clear that
  $\sset{g_{n}}$ is a Borel approximate invariant density for $G$.
  Hence $G$ is Borel amenable as claimed.
\end{proof}

% Recall that a Borel space is standard if it is Borel isomorphic to a
% Borel subset of a Polish space.  A second countable,lhlc\ groupoid
% is the countable union of compact Hausdorff sets.  Hence its
% underlying Borel structure is standard.

Recall that by assumption, all of our Borel spaces, and our Borel
groupoids in particular, are analytic Borel spaces.

\begin{lemma}
  \label{lem-proper}
  Let $G$ be a Borel groupoid acting freely on (the right of) a Borel
  space $X$ such that the quotient map $q:X\to X/G$ has a Borel cross
  section.  Then $X$ is a proper Borel $G$-space.
\end{lemma}
\begin{proof}
  Let $X\rtimes G$ be the transformation groupoid for the right
  $G$-action.\footnote{Here $X\rtimes G=\set{(x,\gamma)\in X\times
      G:s(x)=r(\gamma)}$ has multiplication defined by
    $(x,\gamma)(x\cdot \gamma,\eta)=(x,\gamma\eta)$.}  The map
  $(x,\gamma)\mapsto (x,x\cdot \gamma)$ is a Borel as well as an
  algebraic isomorphism of $X\rtimes G$ onto the equivalence relation
  groupoid $X*_{q}X=\set{(x,y)\in X\times X:q(x)=q(y)}$.  Since
  $X*_{q}X$ and $X\rtimes G$ are both Borel subsets of the
  corresponding product spaces, they are themselves analytic Borel
  spaces.  Hence $X\rtimes G$ and $X*_{q}X$ are Borel isomorphic by
  Corollary 2 of \cite{arv:invitation}*{Theorem~3.3.4}.  Hence it
  suffices to see that $X*_{q}X$ is proper.

  We can identify $X$ with the unit space of $X*_{q}X$.  For each
  $x\in X$, let $m^{x}= \delta_{x}\times \delta_{c(q(x))}$ where
  $c:X/G\to X$ is our Borel cross section for $q$.  Since
  $c(q(x))=c(q(y))$ if $(x,y)\in X*_{q}X$, we have $(y,x)\cdot
  m^{x}=m^{y}$.  Hence $X$ is a proper Borel
  $G$-space.  % The last assertion follows from
  % \cite{arv:invitation}*{Proposition~3.4.2}.
\end{proof}

\section{Cocycles}
\label{sec:cocycles}

In this paper, by a \emph{cocycle} on a groupoid $G$ we mean a 
homomorphism $c:G\to Q$ into a group $Q$.  As in
\cite{ren:groupoid}*{Definition~1.1.9},\footnote{In
  \cite{ren:groupoid}, $G(c)$ is described as pairs $(\gamma,a)\in
  G\times Q$.  The groupoids are isomorphic via the map
  $(a,\gamma,b)\mapsto (\gamma,a)$.} the corresponding
\emph{skew-product} groupoid is
\begin{equation*}
  G(c)=\set{(a,\gamma,b)\in Q\times G\times Q:b=ac(\gamma)}.
\end{equation*}
Then multiplication is given by
$(a,\eta,b)(b,\gamma,d)=(a,\eta\gamma,d)$ and inversion by
$(a,\gamma,b)^{-1}=(b,\gamma^{-1},a)$.  We can identify the unit space
of $G(c)$ with $\go\times Q$, and then the range and source maps are
given as expected: $r(b,\gamma,a)=(r(\gamma),b)$ while
$s(b,\gamma,a)=(s(\gamma),a)$.  If $G$ is a locally Hausdorff, locally
compact groupoid with a Haar system $\set{\lambda^{u}}_{u\in\go}$ and
if $Q$ is a locally compact group, then provided $c$ is continuous,
$G(c)$ is a locally Hausdorff, locally compact groupoid with Haar
system $\lc=\set{\lc^{(u,a)}}_{(u,a)\in \go\times Q}$ where
\begin{equation*}
  \lc^{(u,a)}(g)=\int_{G}g(a,\gamma,ac(\gamma))\,d\lambda^{u}(\gamma).
\end{equation*}

Later it will be useful to note that $Q$ acts on the left of $G(c)$ by
(groupoid) automorphisms:
\begin{equation*}
  q\cdot (b,\gamma,a)=(qb,\gamma,qa).
\end{equation*}
The action of $Q$ on $\go\times Q$, identified with the unit space of
$G(c)$, is given by $q\cdot (u,a)=(u,qa)$.

Part of our interest in $G(c)$ is due to the following result.
\begin{prop}[\cite{ren:groupoid}*{Proposition~II.3.8}]
  \label{prop-ren-groupoids}
  Suppose that $G$ is a second countable \lhlc\ groupoid with a Haar
  system and that $c$ is a continuous homomorphism from $G$ into a
  locally compact group $Q$.  If $G(c)$ and $Q$ are both amenable,
  then so is $G$.
\end{prop}
\begin{proof}
  By \cite{ren:xx13}*{Definition~2.7}, we need to construct a
  topological approximate invariant mean $\sset{g_{n}}$ in $\cc(G)$.
  It will suffice to show that for each pair of compact sets $K\subset
  G$ and $L\subset \go$, and each $\epsilon>0$, there is a
  nonnegative function $f\in \cc(G)$ such that
  \begin{enumerate}[ (A)]
  \item $\displaystyle{\int_{G}f(\gamma)\,d\lambda^{u}(\gamma)\le 1}$
    for all $u\in\go$,
  \item $\displaystyle{\int_{G}f(\gamma)\,d\lambda^{u}(\gamma)}\ge
    1-\epsilon$ for all $u\in L$ and
  \item $\displaystyle{\int_{G} |f(\gamma^{-1}\gamma')-f(\gamma')|\,d
      \lambda^{r(\gamma)}(\gamma')}\le \epsilon$ for all $\gamma\in
    K$.
  \end{enumerate}

  Since $Q$ is amenable and $c(K)$ is compact, there is a nonnegative
  function $k\in C_{c}(Q)$ such that with respect to a fixed
  \emph{right} Haar measure $\alpha$ on $Q$ we have
  \begin{equation*}
    \int_{Q}k(a)\,d\alpha(a)=1\quad\text{and}\quad
    \int_{Q}|k(ab)-k(a)|\,d\alpha(a) \le
    \frac\epsilon2\quad\text{for all $b\in c(K)$.}
  \end{equation*}
  On the other hand, the (topological) amenability of $G(c)$ allows us
  to find a nonnegative function $g\in \cc(G(c))$ such that
  \begin{enumerate}
  \item $\displaystyle{\int_{G}
      g(a,\gamma,ac(\gamma))\,d\lambda^{u}}\le 1$ for all
    $(u,a)\in\go\times Q$,
  \item
    $\displaystyle{\int_{G}g(a,\gamma,ac(\gamma))\,d\lambda^{u}}\ge
    1-\epsilon$ for all $(u,a)\in L\times\supp k$ and
  \item
    $\displaystyle{\int_{G}|g(ac(\gamma),\gamma^{-1}\gamma',ac(\gamma'))
      - g(a,\gamma',ac(\gamma'))| \,d\lambda^{r(\gamma)}}\le \epsilon$
    for all $(a,\gamma)\in\supp k\cdot c(K)^{-1}\times K$.
  \end{enumerate}

  Now we define a nonnegative function on $G$ via
  \begin{equation*}
    f(\gamma):=\int_{G} k(a)g(a,\gamma,ac(\gamma))\,d\alpha(a).
  \end{equation*}
  If $g$ is continuous and supported in a compact Hausdorff subset,
  then $f$ is too.  Since in general, $g$ a finite sum of such
  functions, we have $f\in \cc(G)$.  Then it is easy to check $f$
  satisfies (A)~and (B) above.  On the other hand,
  \begin{align*}
    \int_{G}& |
    f(\gamma^{-1}\gamma')-f(\gamma')|\,d\lambda^{r(\gamma)}(\gamma')   \\
    &=\int_{G}\int_{Q}k(a)\bigl(
    g(a,\gamma^{-1}\gamma',ac(\gamma^{-1}\gamma')) -
    g(a,\gamma',ac(\gamma')) \bigr)\,d\lambda^{r(\gamma)}(\gamma') \\
    & = \int_{G}\int_{Q} k(ac(\gamma)) \bigl(
    g(ac(\gamma),\gamma^{-1}\gamma',a(c(\gamma')) -
    g(a,\gamma',ac(\gamma') ) \bigr)
    \,d\lambda^{r(\gamma)}(\gamma')\,d\alpha(a) \\ 
    & \hskip1in + \int_{G}\int_{Q} \bigl( k(ac(\gamma))-k(a)\bigr)
    g(a,\gamma',ac(\gamma'))
    \,d\lambda^{r(\gamma)}(\gamma')\,d\alpha(a)
    \\
    \intertext{which, in view of our assumptions on $k$ and $g$, is}
    &\le \frac\epsilon2+\frac\epsilon2=\epsilon.\qedhere
  \end{align*}

\end{proof}

\begin{thm}
  \label{thm-skewprod}
  Suppose that $G$ is a \lhlc\ second countable groupoid with a Haar
  system and that $c:G\to Q$ is a continuous homomorphism into a
  locally compact group $Q$ such that the kernel $c^{-1}(e)$ is
  amenable.  Let $\ts:G\to \go\times Q$ be the map given by
  $\ts(\gamma)=(s(\gamma),c(\gamma))$.  Let $Y$ be the image of $\ts$
  in $\go\times Q$.  Then $Y$ is $G(c)$-invariant.  Suppose that 
    there are countably many $q_{n}\in Q$ such
    that
    \begin{equation*}
      \go\times Q=\bigcup_{n} q_{n}\cdot Y,
    \end{equation*}
    then $G(c)$ is amenable.  In particular, if $Y$ is open in
    $\go\times Q$, then $G(c)$ is amenable.
  If in addition, $Q$ is amenable, then so
  is $G$.
\end{thm}

It is not hard to check that $Y$ is invariant: if
$s(a,\gamma,ac(\gamma))=(s(\gamma),ac(\gamma))\in Y$, then there is an
$\eta\in G$ such that $\ts(\eta)=(s(\gamma),ac(\gamma))$.  Hence
$s(\eta)=s(\gamma)$ and $c(\eta)=ac(\gamma)$.  But then
$\ts(\eta\gamma^{-1})=
(r(\gamma),c(\eta)c(\gamma)^{-1})=(r(\gamma),a)=r(a,\gamma,a
c(\gamma))$.

Our key tool is the following observation.

\begin{prop}
  \label{prop-gy-equi-ker}
  Let $G$, $Q$, $c$ and $Y=\ts(G)$ be as in the statement of
  Theorem~\ref{thm-skewprod}.  Then the Borel groupoids $G(c)\restr Y$
  and $c^{-1}(e)$ are Borel equivalent.
\end{prop}
\begin{remark}
  \label{rem-complicated}
  The proof is complicated by the fact that we are not assuming $G$ is
  Hausdorff, nor is it clear whether $c^{-1}(e)$ has a Haar system.
  Thus we cannot apply \cite{anaren:amenable00}*{Corollary~2.1.17} to
  establish that the actions are proper.  Instead, we resort to the
  definition.
\end{remark}

\begin{proof}
  We want to show that $G$ is a $(c^{-1}(e),G(c)\restr Y)$-Borel
  equivalence.  We let $c^{-1}(e)$ act on the left by multiplication
  in $G$.  (Hence the moment map  for the left action is just
  $r$.)  For the right action of $G(c)\restr Y$ we use the moment map
  $\ts$.  Thus $\eta\cdot (a,\gamma,ac(\gamma))$ makes sense only when
  $(\eta,\gamma)\in G^{(2)}$ and $c(\eta)=a$.  Then we define
  \begin{equation}
    \label{eq:1}
    \eta\cdot (c(\eta),\gamma,c(\eta\gamma)) =\eta\gamma.
  \end{equation}

  To see that $G$ is a proper right $G(c)\restr Y$-space, we need to
  see that the transformation groupoid $G\rtimes G(c)\restr Y$ for the
  right $G(c)\restr Y$-action is proper.  But the map
  \begin{equation*}
    (\eta,(c(\eta),\gamma,c(\eta\gamma))\mapsto (\eta,\gamma)
  \end{equation*}
  is clearly a bijection of $G\rtimes G(c)\restr Y$ onto $G\rtimes G$.
  It is easy to see that it is a groupoid homomorphism as well.  Since
  the right action of any Borel groupoid on itself is proper (by a
  left to right modification of
  \cite{anaren:amenable00}*{Example~2.1.4(1)}), it follows that $G$ is
  a proper $G(c)\restr Y$-space.  The action is clearly free.

  To see that the left action is proper in the Borel sense, we first
  note that it is certainly free and proper topologically.  Hence the
  orbit space is itself a \lhlc\ space by
  \cite{muhwil:nyjm08}*{Lemma~2.6}.  Hence it is a standard Borel
  space with the Borel structure coming from its topology. Let $q:G\to
  c^{-1}(e)\backslash G$ be the quotient map.  Since $q$ is continuous
  and $c^{-1}(e)$ is closed, the inverse image of points in
  $c^{-1}(e)\backslash G$ under $q$ are closed.  Since open sets in
  $G$ are $\sigma$-compact, the forward image of open sets are
  $\sigma$-compact.  Since a compact subset of a \lhlc\ space is
  Borel, so is the forward image of any open set.  Hence $q$ has a
  Borel cross section $w$ by \cite{arv:invitation}*{Theorem~3.4.1}.
  \emph{A priori} $c^{-1}(e)\backslash G$ has two Borel structures:
  the one we've been using coming from the quotient topology, and also
  the quotient Borel structure.  The quotient map $q$ is Borel in both
  cases and $w$ must also be a cross section for the potentially finer
  quotient Borel structure.  Hence both spaces are Borel isomorphic to
  the image of $w$ by \cite{arv:invitation}*{Proposition~3.4.2}.
  Hence the two Borel structures coincide on the orbit space.
  Therefore Lemma~\ref{lem-proper} applies and $G$ is a free and
  proper Borel $c^{-1}(e)$-space as required.

  If $\eta\in c^{-1}(e)$, then
  $\ts(\eta\gamma)=(s(\gamma),c(\eta\gamma))=(s(\gamma),c(\gamma))$.
  It follows that $\ts$ factors through $G/c^{-1}(e)$.  On the other
  hand, suppose that
  \begin{equation*}
    \ts(\gamma)=(s(\gamma),c(\gamma))=\ts(\eta)=(s(\eta),c(\eta)).
  \end{equation*}
  Then $c(\gamma)=c(\eta)$ and $\gamma^{-1}\eta\in c^{-1}(e)$.  Of
  course $\gamma\cdot \gamma^{-1}\eta=\eta$.  Thus $\ts$ induces a
  bijection of $G/c^{-1}(e)$ with $Y$.  Since $c^{-1}(e)$ is closed,
  it acts freely and properly on $G$.  Thus the quotient
  $G/c^{-1}(e)$ is a standard Borel space (with respect to the
  quotient Borel structure) and $\ts$ induces a Borel bijection of
  $G/c^{-1}(e)$ with the analytic Borel space $Y$.  Hence $\ts$ is a
  Borel isomorphism by \cite{arv:invitation}*{Corollary 2 of
    Theorem~3.3.4}.

  It is clear from \eqref{eq:1} that $\trr=r$ factors through
  $G/G(c)\restr Y$.  In fact, the homeomorphism of $G\rtimes
  G(c)\restr Y$ with $G\rtimes G$ induces a homemorphism of
  $G/G(c)\restr Y$ onto $G/G\cong \go$.  Hence $\trr$ certainly
  induces a Borel isomorphism as required.
\end{proof}

\begin{proof}[Proof of Theorem~\ref{thm-skewprod}]
  If $Y$ is open, then $\sset{q\cdot Y}_{q\in Q}$ is an open cover of
  the second countable set $\go\times Q$.  By a theorem of Lindel\"of,
  it has a countable
  subcover; so it suffices to prove the first statement. 

  Since $G$ is $\sigma$-compact
  and $\ts$ is continuous, $Y$ is also $\sigma$-compact and hence
  Borel.  This ensures that $G(c)\restr Y$ is standard as a Borel
  space. Since $c^{-1}(e)$ is assumed to be amenable, it is Borel
  amenable.  Hence $G(c)\restr Y$ is Borel amenable by
  Proposition~\ref{prop-gy-equi-ker}.  On the other hand $G(c)\restr
  {q_{n}\cdot Y}=q_{n}\cdot \bigl(G\restr Y\bigr)$.  Hence each
  $G(c)\restr {q_{n}\cdot Y}$ is Borel amenable.  Since $Q$ acts by
  automorphisms, each $q_{n}\cdot Y$ is invariant.  Now $G(c)$ is
  Borel amenable from Lemma~\ref{lem-borel-cover}.  Since $G(c)$ has a
  Haar system if $G$ does, it follows from
  \cite{ren:xx13}*{Corollary~2.15} that $G$ is amenable.

  The last assertion follows from
  Proposition~\ref{prop-ren-groupoids}.
\end{proof}

As special cases of Theorem~\ref{thm-skewprod}, we obtain the cocycle
results from \cite{anaren:amenable00} and \cite{spi:tams14} mentioned
in the introduction.  The following result strengthens
\cite{spi:tams14}*{Theorem~9.3} by removing the hypotheses that $G$ be
\'etale. 

\begin{cor}
  \label{cor-discrete}
  Suppose that $G$ is a \lhlc\ second countable groupoid with a Haar
  system and that $c:G\to Q$ is a continuous homomorphism into a
  discrete group $Q$ such that $c^{-1}(e)$ is amenable.  Then $G(c)$
  is amenable.  If $Q$ is amenable, then so is $G$.
\end{cor}
\begin{proof}
  It suffices to see that $Y=\ts(G)$ is open in $\go\times Q$.  But
  $s$ is open and
  \begin{equation*}
    \ts(G)=\bigcup_{a\in Q} s(c^{-1}(a))\times \sset a,
  \end{equation*}
  which is open since each $c^{-1}(a)$ is open.
\end{proof}

 The result
  \cite{anaren:amenable00}*{Theorem~5.3.14} mentioned in the
  introduction establishes the amenability of a measured groupoid $G$
  which admits a strongly surjective Borel homomorphism onto a Borel
  groupoid with amenable kernel and range.  The following result is
  similar, but concerns topological instead of measure amenability.
  Moreover, it is more restrictive since it assumes that the range is
  a group rather than a groupoid.
\begin{cor}
  \label{cor-str-surj}
  Suppose that $G$ is a \lhlc\ second countable groupoid with a Haar
  system and that $c:G\to Q$ is a strongly surjective continuous
  homomorphism into a group $Q$ such that $c^{-1}(e)$ is amenable.
  Then $G(c)$ is amenable.  If $Q$ is amenable, then so is $G$.
\end{cor}
\begin{proof}
Since $c$ is strongly surjective,  we must have $c(G_{u})=Q$ for
  all $u$. It follows that $\ts$
  is surjective, and the result follows from
  Theorem~\ref{thm-skewprod}.
\end{proof}

\section{Application to Semigroup Actions}
\label{sec:semigroup action}

Our results apply nicely to semigroup action groupoids.

\begin{definition}
  Let $X$ be a set and $P$ be a semigroup with identity $e$. A
  \emph{right (partial) action of $P$ on $X$} consists of a subset
  $X*P$ of $X\times P$ and a map $T: X* P \to X$
   sending $(x,m)$ to $x\cdot
  m$, such that
  \begin{enumerate}
  \item for all $x \in X$, $(x,e)\in X*P$ and $x\cdot e=x$;
  \item for all $(x,m,n)\in X\times P\times P$, $(x,mn)\in X*P$ if and
    only if $(x,m)\in X*P$ and $(x\cdot m,n)\in X*P$; if this holds,
    $(x\cdot m) \cdot n=x\cdot (mn)$.
  \end{enumerate}
  For all $m\in P$, we define $U(m)=\{x: (x,m)\in X*P\}$ and $V(m)=
  \{x\cdot m: (x,m)\in X*P\}$ and $T_m:U(m)\to V(m)$ such that
  $T_mx=x\cdot m$.  The triple $(X,P,T)$ will be called a
  \emph{semigroup action}.
\end{definition}

This notion generalizes that of singly generated dynamical system
(SGDS) in the sense of \cite{ren:otm00}, which is the case $P=\N$. A
SGDS is given by a single map $T$ from a subset
$\operatorname{dom}(T)$ of $X$ to another subset
$\operatorname{ran}(T)$ of
$X$. For $n\in\N$, we let $U(n)=\operatorname{dom}(T^n)$ and define
$x\cdot n=T^n(x)$ for $x\in U(n)$.\\

Let us define the binary relation on $P$: $m\le m'$ if and only if
there exists $n\in P$ such that $m'=mn$. Our axioms imply that $m\le
m'\Rightarrow U(m')\subset U(m)$.

The next notion is important because it will allow a simple
construction of a groupoid from a semigroup action.

\begin{definition} \label{def-directed}
  Let us say that a semigroup action $(X,P,T)$ is \emph{directed} if
  for all pairs $(m,n)\in P\times P$ such that $U(m)\cap U(n)$ is
  non-empty, there exists an upper bound $r\in P$ of $m$ and $n$ such
  that $U(m)\cap U(n)=U(r)$.
\end{definition}
Note that the equality can be replaced by the inclusion $U(m)\cap
U(n)\subset U(r)$. Here are two rather different situations where the
action is directed.
\begin{example}
  Assume that the action is everywhere defined, i.e.  $X*P=X\times
  P$. Then the action is directed if and only if the semigroup $P$ is
  directed, in the sense that for all $m,n\in P$, there exists $r\in
  P$ such that $m,n\le r$; we then say that $r$ is a common upper
  bound, or c.u.b.\ for short, of $m$ and $n$. If $P$ is a
  subsemigroup of a group $Q$, meaning that
  \begin{equation}\label{eq:subsemi}
    P\subset Q, \quad PP\subset P \quad \text{and}\quad e\in P.
  \end{equation}
  then, $P$ is directed if and and only if it satisfies the \emph{Ore
    condition} $P^{-1}P\subset PP^{-1}$. In such a case, $P$ is called
  an \emph{Ore semigroup}.
\end{example}

\begin{example}\label{quasi-lattice ordered} Assume that $P$ is
  \emph{quasi-lattice ordered}. This means that $P$ is a subsemigroup
  of a group $Q$ such that $P\cap P^{-1}=\{e\}$ and whenever two
  elements $m,n\in P$ have a c.u.b. in $P$, they have a least upper
  bound, or l.u.b.\ for short, $m\vee n$ (the original definition of
  A. Nica in \cite{nic:jot92} contains the additional assumption that
  every element of $Q$ which has an upper bound in $P$ has a least
  upper bound in $P$). Then, the action of $P$ on $X=P$ given by
  \begin{equation*}
    U(n)=\set{x\in P: n\le x}\quad \text{and}\quad x\cdot n=n^{-1}x
  \end{equation*}
  is directed.
\end{example}
 
If $(X,P,T)$ is directed, then the relation $x\sim y$ if and only if
there exist $(m,n)\in P\times P$ such that $x\cdot m=y\cdot n$ is
transitive and hence an equivalence relation. We denote by $[x]$ the
equivalence class of $x$.

Given $P$ a subsemigroup of a group $Q$ and a directed semigroup
action $(X,P,T)$, we define $G(X,P,T)$ as the set of triples $(x,q,y)$
in $X \times Q\times X$ such that there exist $m,n\in P$ with
$q=mn^{-1}$, $x\in U(m)$, $y\in U(n)$ and $x\cdot m=y\cdot n$.

\begin{lemma} Assume that $P$ is a subsemigroup of a group $Q$ and
  that $(X,P,T)$ is a directed action. Then $G(X,P,T)$ is a
  subgroupoid of $X\times Q\times X$, equipped with its natural
  structure of a groupoid over $X$.
\end{lemma}

\begin{proof}
  Suppose that $(x,q,y)$ belongs to $G$: there exist $m,n\in P$ such
  that $q=mn^{-1}$ and $x\cdot m=y\cdot n$. Then $(y,q^{-1},x)$ also
  belongs to $G$ because $q^{-1}=nm^{-1}$ and $y\cdot n=x\cdot
  m$. Suppose that $(x,s,y)$ and $(y,t,z)$ belong to $G$: there exist
  $m,n,p,q\in P$ such that $s=mn^{-1}, t=pq^{-1}$, $x\cdot m=y\cdot n$
  and $y\cdot p=z\cdot q$. Since the action is directed, there exist
  $(a,b)\in P\times P$ such that $na=pb$. Then $x\cdot ma=y\cdot
  na=y\cdot pb=z\cdot qb$. Since $(ma)(qb)^{-1}=(mn^{-1})(pq^{-1})$,
  $(x,st,y)$ belongs to $G$.
\end{proof}
\begin{definition} Let $(X,P,T)$ be a directed semigroup action, where
  $P$ is a subsemigroup of a group $Q$. The groupoid $G(X,P,T)$
  associated with $(X,P,T)$ is called the \emph{semidirect product
    groupoid} (or semidirect product for short) of the action. It
  carries the canonical cocycle $c:G(X,P,T)\to Q$ given by
  $c(x,q,y)=q$.
\end{definition}

\begin{remark} When $X$ is reduced to a point, the lemma says that
  $PP^{-1}$ is a group if $P$ satisfies the Ore condition
  $P^{-1}P\subset PP^{-1}$; clearly, this condition is also necessary.

\end{remark}

% The case when $Q$ is a locally compact group is left for the future.
% We next make topological assumptions: we assume that $X$ is locally
% compact %and Hausdorff, that $Q$ is a locally compact group, that
% $X*P$ is an open %subset of $X\times P$ and that $T:X*P\to X$ is
% continuous. Given open sets %$U,V\subset X$ and $A,B\subset P$, we
% define
%$$Z(U,A,B,V):=\{(x,mn^{-1},y)\in G: x\in U, y\in V, m\in A, n\in B\}$$

We next make the following topological assumptions:

\begin{definition} We shall say that a semigroup action $(X,P,T)$,
  where $P$ is a subsemigroup of a group $Q$, is \emph{locally
    compact} if
  \begin{enumerate}
  \item $X$ is a locally compact Hausdorff space;
  \item $Q$ is a discrete group;
  \item for all $m\in P$, $U(m)$ and $V(m)$ are open subsets of $X$
    and $T_m: U(m)\to V(m)$ is a local homeomorphism.
  \end{enumerate}
\end{definition}

In the following, we always assume that the semigroup action $(X,P,T)$
is locally compact.

Given $m,n\in P$ and $A,B$ subsets of $X$, we define
\begin{equation*}
  Z(A,m,n,B):=\set{(x,mn^{-1},y)\in G: \text{$x\in A$, $y\in B$ and $ x\cdot m=y\cdot n$}}
\end{equation*}
Let $\mathcal B$ be the family of subsets $Z(U,m,n,V)$, where $U$ and
$V$ are open subsets of $X$.

\begin{lemma} The family $\mathcal B$ is a base for a topology
  $\mathcal T$ on $G$.
\end{lemma}

\begin{proof} First, it is clear that this family covers $G$. Second,
  let $(x,q,y)$ be a point in the intersection of $Z(U,m,n,V)$ and
  $Z(U',m',n',V')$. We are going to find $Z(U'',m'',n'',V'')$
  containing the point $(x,q,y)$ and contained in the
  intersection. There exists $(a,a')\in P\times P$ such that $m
  a=m'a'$ is an upper bound $r$ of $m$ and $m'$ and $U(r)=U(m)\cap
  U(m')$. Then $na=n'a'$ is an upper bound $s$ of $n$ and $n'$ and
  $U(s)=U(n)\cap U(n')$. Let $W$ be an open neighborhood of $x\cdot
  m=y\cdot n$ on which $T_a$ is injective. Similarly, let $W'$ be an
  open neighborhood of $x\cdot m'=y\cdot n'$ on which $T_{a'}$ is
  injective. We let $U''=U\cap U'\cap T_m^{-1}(W)\cap
  T_{m'}^{-1}(W')$, $V''=V\cap V'\cap T_n^{-1}(W)\cap
  T_{n'}^{-1}(W')$, $m''=r$ and $n''=s$. Since $x\in U''$, $y\in V''$
  and $x\cdot r=y\cdot s$, $(x,q,y)$ belongs to
  $Z(U'',m'',n'',V'')$. If $(x',q,y')$ belongs to
  $Z(U'',m'',n'',V'')$, $x'\in U\cap U'$, $y'\in V\cap V'$ and
  $x'\cdot ma=y'\cdot na$. Since $x'\cdot m$ and $y'\cdot n$ belong to
  $W$ on which $T_a$ is injective, $x'\cdot m=y'\cdot n$. Therefore
  $(x',q,y')$ belongs to $Z(U,m,n,V)$. Similarly, the equality
  $x'\cdot m'a'=y'\cdot n'a'$ implies the equality $x'\cdot m'=y'\cdot
  n'$ because $x'\cdot m'$ and $y'\cdot n'$ belong to $W'$ on which
  $T_{a'}$ is injective.
\end{proof}

\begin{remark}
  The sets $Z(U,m,n,V)$, where $U$ and $V$ are open subsets of $X$
  such that $U\subset U(m), V\subset U(n)$ and ${T_m}_{|U}$ and
  ${T_n}_{|V}$ injective form a subbase of $\mathcal B$.
\end{remark}

\begin{lemma} \label{lem-above} The topology $\mathcal T$ of $G$ is
  finer than the product topology of $X\times Q\times X$ but it agrees
  with it on the sets $Z(A,m,n,B)$, where $A,B$ are subsets of $X$ and
  $m,n\in P$.
\end{lemma}

\begin{proof} The intersection of a rectangle $U\times\{q\}\times V$
  with $G$, where $U,V$ are open subsets of $X$ is a union of elements
  of $\mathcal B$. Therefore, the topology $\mathcal T$ is finer than
  the product topology. Let $A,B$ be subsets of $X$ and $m,n\in P$.
  Let $(x,q,y)$ be a point of $Z(A,m,n,B)$. Let $Z(U,m',n',V)\cap
  Z(A,m,n,B)$ be a basic neighborhood of $(x,q,y)$ in
  $Z(A,m,n,B)$. Just as in the proof of the previous lemma, we
  introduce $(a,a')\in P\times P$ such that $ma=m'a'$ (denoted by $r$)
  and $U(r)=U(m)\cap U(m')$, $s=na=n'a'$, and we choose and open
  neighborhood $W$ of
  $x\cdot m=y\cdot n$ on which $T_a$ is injective, and an open
  neighborhood $W'$  of $x\cdot m'=y\cdot n'$ on which $T_{a'}$ is
  injective. We define $U'=U\cap T_m^{-1}(W)\cap T_{m'}^{-1}(W')$ and
  $V'=V\cap T_n^{-1}(W)\cap T_{n'}^{-1}(W')$. Suppose that $(x',q,
  y')$ belongs to $(U'\times\{q\}\times V')\cap Z(A,m,n,B)$. Then
  $x'\cdot m'a'=x'\cdot ma=y'\cdot na=y'\cdot n'a'$. We deduce as
  before that $x'm'=y'n'$. Therefore $(x',q, y')$ belongs to
  $Z(U,m',n',V)$.
\end{proof}

\begin{prop}
  \label{prop-semidirect-product-groupoid} Let $(X,P,T)$ be a
  semigroup action as above. We endow $G(X,P,T)$ with the topology
  $\mathcal T$. Then,
  \begin{enumerate}
  \item $G(X,P,T)$ is an \'etale locally compact Hausdorff groupoid;
  \item the canonical cocycle $c: G(X,P,T)\to Q$ is continuous.
  \end{enumerate}
\end{prop}
\begin{proof} The topology $\mathcal T$ is Hausdorff because it is
  finer than the product topology. Let $(x,q,y)\in G$. Pick $(m,n)\in
  P\times P$ such that $q=mn^{-1}$ and $x\cdot m=y\cdot n$. Let $A,B$
  be compact neighborhoods of $x, y$ contained respectively in $U(m)$
  and in $U(n)$. Then $Z(A, m, n, B)$ is a compact neighborhood of
  $(x,q,y)$, and $\mathcal T$ is a locally compact topology. The
  injection map $i(x)=(x,e,x)$ is a homeomorphism from $X$ onto
  $G^{(0)}$.

  The inverse map $(x,q,y)\mapsto (y,q^{-1},x)$ transforms
  $Z(U,m,n,V)$ into $Z(V,n,m,U)$. Therefore, it is a
  homeomorphism. Suppose that $(x_\alpha, s_\alpha,y_\alpha)$
  converges to $(x,s,y)$ and $(y_\alpha, t_\alpha,z_\alpha)$ converges
  to $(y,t,z)$. Pick basic open sets $Z(U,m,n,V)$ and $Z(V',p,q,W)$
  containining respectively $(x,s,y)$ and $(y,t,z)$. By definition,
  for $\alpha$ large enough $(x_\alpha, s_\alpha,y_\alpha)$ belongs to
  $Z(U,m,n,V)$ and $(y_\alpha, t_\alpha,z_\alpha)$ belongs to
  $Z(V',p,q,W)$. As in Lemma~\ref{lem-above}, pick $(a,b)\in P\times
  P$ such that $na=pb$. Then $(x,st,z)$ belongs to $Z(U,ma,qb,W)$ and
  so does $(x_\alpha,s_\alpha t_\alpha,z_\alpha)$ for $\alpha$ large
  enough. We apply Lemma~\ref{lem-above} to conclude that
  $(x_\alpha,s_\alpha t_\alpha,z_\alpha)$ converges to
  $(x,st,z)$. Thus $G$ is a topological locally compact Hausdorff
  groupoid.

  A basis element $S=Z(U,m,n,V)$ such that $T_m$ is injective on $U$
  and $T_n$ is injective on $V$ is a bisection, in the sense that the
  restrictions of the range and source maps $r_{|S}$ and $s_{|S}$ are
  homeomorphisms onto open subsets of $X$. Since $G$ admits a cover of
  open bisections, it is an \'etale groupoid.

  Since the canonical cocycle is continuous with respect to the
  product topology, it is continuous with respect to $\mathcal T$.
\end{proof}

Here is our main application of our Theorem 4.2 (in the form of
  Corollary 4.5); it gives the amenability of the semidirect product
  by a subsemigroup of an amenable group.

\begin{thm} \label{thm-app-directed-sg-actions} Let $(X,P,T)$ be
  a directed locally compact semigroup action. Assume that $P$ is
  a subsemigroup of a countable amenable group $Q$. Then the
  semi-direct product groupoid $G(X,P,T)$ is topologically amenable.
\end{thm}
  \begin{proof}
  To prove this result, we apply Corollary~\ref{cor-discrete} to the
  continuous cocycle $c:G(X,P,T)\to Q$. The only missing point is the
  amenability of the equivalence relation $R=c^{-1}(e)$.  As is
  standard, we try to write $R$ as the increasing union of
  well-behaved equivalence relations (for example, see
  \cite{simwil:xx15}*{Lemma~3.5}).

  To see how to do this, we call a subset $F\subset P$
  \emph{action-directed} if $e\in F$ and given $n,m\in F$ with
  $U(n)\cap U(m)\not=\emptyset$, then there is an $r\in F$ dominating
  $n$ and $m$ such that $U(r)=U(n)\cap U(m)$.  For example, by
  hypotheses, $P$ itself is action-directed.  More to the point, if
  $F$ is action-directed, then
  \begin{equation*}
    R_{F}=\set{(x,y):\text{there is a $m\in F$ such that $x\cdot
        m=y\cdot m$}}
  \end{equation*}
is a Borel equivalence relation: it is an $F_{\sigma}$ subset of
$X\times X$ and an equivalence relation since $F$ is action-directed.
We just need to specify suitable sets $F$.

Let $\F$ be the collection of finite subsets $F$ of $P$ such that
$\bigcap_{m\in F}U(m)\not=\emptyset$.
A a simple
  induction argument implies the following.
  \begin{claim}
    \label{cl-1}
If $F\in\F$, then there is an
$r\in P$ such that $n\le r$ for all $n\in F$ and
$U(r)=\bigcap_{n\in F}U(n)$.
  \end{claim}
  \begin{claim}
    \label{cl-2}
There is a map $F\mapsto r_{F}$ from
  $\F$ to $P$ such that $r_{\emptyset}=e$, 
$r_{\{n\}}=n$, and such that
  \begin{enumerate}
  \item $n\le r_{F}$ for all $n\in F$,
  \item $U(r_{F})=\bigcap_{n\in F}U(n)$, and
  \item $F'\subset F$ implies $r_{F'}\le r_{F}$.
  \end{enumerate}
  \end{claim}
  \begin{proof}[Proof of Claim~\ref{cl-2}]
    We start by defining $r_{\emptyset}$ and $r_{\{n\}}$ as above.
    Suppose that we have defined $r_{F}$ for all $F\in\F$ with $|F|\le
    k$ for some $k\ge 1$ such that (a), (b) and (c) hold.

   Let us define $r_{F}$ for $F\in\F$ with $k+1$ elements.  Note that
   \begin{equation*}
     \F'=\set{S\subset F:|S|=k}
   \end{equation*}
is a subset of $\F$.  By Claim~\ref{cl-1}, we can define $r_{F}\in P$
so that $r_{S}\le r_{F}$ for all $S\in\F'$ and such that
$U(r_{F})=\bigcap_{S\in\F'} U(r_{S})$.

Now consider the set $\set{r_{F}:\text{$F\in \F$ and $|F|\le k+1$}}$.
Then (a) and~(b) hold by assumption if $|F|\le k$.  But if $|F|=k+1$
and $n\in F$, then there is an $S\in\F'$ such that $n\in S$.  Hence
$n\le r_{S}\le r_{F}$.  Similarly,
\begin{equation*}
  U(r_{F})=\bigcap_{S\in\F'} U(r_{S}) = \bigcap_{S\in\F'}\bigcap_{n\in
  S} U(n) =\bigcap_{n\in F}U(n).
\end{equation*}
Hence (a) and~(b) hold for sets of $k+1$ or fewer elements.

Now suppose $F'$ is a proper subset of $F$.  We have $r_{F'}\le r_{F}$
by assumption if $|F|\le k$.  If $|F|=k+1$, then there exists  $S\subset F$
with $|S|=k$.  Then $r_{F'}\le r_{S}\le r_{F}$.
  \end{proof}

  As pointed out by the referee, the possibility that $\F$ might be
  directed is suggested by \cite{lacrae:jfa96} where the authors work with
  quasi-lattice ordered subgroups.% \footnote{In an earlier approach, we
    % looked at ``finite intervals'' $F_{m}=\set{p:e\le p \le m}$.
    % However, these intervals are not action-directed without
    % additional hypotheses.  More troublesome, they are not necessarily
    % finite sets.} 
  That the same is true for directed actions is
  implied by the following claim.
  \begin{claim}
    \label{cl-3}
Every finite subset $S\subset P$ is contained in a \emph{finite}
action-directed subset $F$.
  \end{claim}
  \begin{proof}[Proof of Claim~\ref{cl-3}]
    Let $F\mapsto r_{F}$ be the map defined in Claim~\ref{cl-2}. Define
  \begin{equation*}
    F=\set{ r_{S'} :\text{$S'\subset S$ and $S'\in\F$}}.
  \end{equation*}
  Clearly $e\in F$ and $S\subset F$.  We claim that $F$ is
  action-directed.  Let $k,l\in F$ be such that $U(k)\cap
  U(l)\not=\emptyset$.  We can assume that $k=r_{S'}$ and $l=r_{S''}$
  for appropriate subsets of $S$.  Then $F'=S'\cup S''$ satisfies
  $\bigcap_{n\in F'}U(n)=U(k)\cap U(l)\not= \emptyset$.  But then
  $r_{F'}\in F$.  Since $F\mapsto r_{F}$ is monotonic, we have
  $r_{S'}\le r_{F'}$ and $r_{S''}\le r_{F'}$.  This completes the
  proof of the claim.
  \end{proof}

  \begin{claim}
    \label{cl-4}
  There is a sequence $(F_{i})$ of finite action-directed sets such
  that $F_{i}\subset F_{i+1}$ and such that
  \begin{equation*}
    c^{-1}(e) =\bigcup_{i}R_{F_{i}}.
  \end{equation*}
  \end{claim}
  \begin{proof}
    [Proof of Claim~\ref{cl-4}]
    Let $P=\sset{p_{1},p_{2},\dots}$ with $p_{1}=e$.  Now we can
    employ Claim~\ref{cl-3} to inductively construct the $F_{i}$ where
    $F_{i+1}$ is an action-directed set containing $F_{i}$ and $p_{i+1}$.
  \end{proof}

  The key observation is that if $F$ is finite and action-directed,
  then $R_{F}$ is proper. To see this, note that the equivalence class
  $[x]_{F}$ is the finite union over $m\in F$ of the fibres
  $T_{m}^{-1}(x\cdot m)$.  Since $T_{m}$ is a local homeomorphism, the
  fibres are discrete.  Hence each orbit is discrete and therefore
  locally closed.  The Mackey-Glimm-Ramsay dichotomy \cite{ram:jfa90}
  then implies that the orbit space $X/R_{F}$ is a standard Borel
  space. Then \cite{anaren:amenable00}*{Example~2.1.4(2)} implies that
  $R_{F}$ is a proper Borel groupoid.

It follows from Claim~\ref{cl-4} and
Lemma~\ref{lem-incr-union} that $c^{-1}(e)$ is Borel amenable.  Since
it is open in $G(X,P,T)$, it is \'etale.  Hence it is amenable by
\cite{ren:xx13}*{Corollary~2.15}. 
\end{proof}

\begin{remark}
  It is well known that there are interesting amenable actions (in the
  sense that the semi-direct product groupoid of the action is
  amenable) of non-amenable groups. We will encounter in the next
  section an amenable action of a free semigroup (appearing in the
  work of Nica \cite{nic:jot92} on the Wiener-Hopf algebra of a
  semigroup). In this case, the amenability of the action cannot be
  deduced from the above theorem as the free group is not amenable.
\end{remark}

\section{Application to Topological Higher Rank Graphs}
\label{sec:THRG}

Higher-rank graphs provide interesting semigroup actions which
generalize one-sided subshifts of finite type. We recall some
definitions but refer to \cite{yee:jot07} for a complete
exposition. % With respect to the previous literature on higher rank
% graphs, we introduce two novelties. First, we define $P$-graphs for an
% arbitrary subsemigroup $P$ of a group $Q$ while the original
% definition is limited to the case $P=\N^d\subset Q=\Z^d$. However, to
% develop the theory smoothly, we shall often need the assumption that
% $P$ is quasi-lattice ordered, as defined in in
% Example\,\ref{quasi-lattice ordered}. Second, we construct the path
% space $\Omega$ as a closure in the space of closed subsets of the
% higher rank graph $\Lambda$ with respect to the Fell topology. 

% \dpw{We should mention that employing more general subgroups $P$ was
% discussed briefly in \cite{kumpas:nyjm00}, and discrete $P$-graphs were
% considered in \cite{bsv:ijm13}.  Generalizations of paths in terms of
% filters appear in \cite{crilac:jfa07}, and a similar idea appears in
% \cite{exe:sf09} where the closure of the ultrafilters in the
% idempotent semilattice of an inverse semigroup shows up as the ``tight
% spectrum'' of the inverse semigroup.  This idea is also present in
% \cite{spi:tams14}, and filters and ultrafilers in $P$-graphs were also
% used in \cite{bsv:ijm13}.  
% }
We introduce two changes with respect to \cite{yee:jot07}. First, we
define $P$-graphs for an arbitrary subsemigroup $P$ of a group $Q$
while Yeend considers the case $P=\N^d\subset Q=\Z^d$. In order to
develop the theory smoothly, we shall often need the assumption that
$P$ is quasi-lattice ordered, as defined in Example~\ref{quasi-lattice
  ordered}. Such $P$-graphs have already been introduced in
\cite{bsv:ijm13}. Second, we construct the path space $\Omega$ as a
closure in the space of closed subsets of the $P$-graph with respect
to the Fell topology. The use of directed hereditary subsets to
construct the path space goes back to \cite{nic:jot92} and is present
in \cite{exe:sf09,bsv:ijm13,spi:tams14}.

Here
are our definitions. A \emph{small category} $\Lambda$ is given by its
set of arrows $\Lambda^{(1)}$ (usually denoted by $\Lambda$), its set
of vertices $\Lambda^{(0)}$ (viewed as a subset of $\Lambda$ through
the identity map $i:\Lambda^{(0)}\to\Lambda$), range and source maps
$r,s:\Lambda\to\Lambda^{(0)}$ and composition map $\circ:
\Lambda^{(2)}\to\Lambda$ where $\Lambda^{(2)}$ is the set of
composable pairs of arrows, i.e.,
$(\lambda,\mu)\in\Lambda\times\Lambda$ such that
$s(\lambda)=r(\mu)$. Given $A,B\subset\Lambda$, we write $A*B=(A\times
B)\cap \Lambda^{(2)}$. We make the following topological assumptions.
\begin{enumerate}
\item $\Lambda$ and $\Lambda^{(0)}$ are locally compact Hausdorff
  spaces;
\item $r,s:\Lambda\to \Lambda^{(0)}$ are continuous and $s$ is a local
  homeomorphism;
\item $i:\Lambda^{(0)}\to\Lambda$ is continuous;
\item composition $\circ: \Lambda^{(2)}\to\Lambda$ is continuous and
  open.
\end{enumerate}

The following definition is a topological version of
\cite{bsv:ijm13}*{Definition 2.1}. 
\begin{definition} Let $P$ be a semigroup with unit element $e$. A
  \emph{higher-rank topological graph graded by $P$}, or
  \emph{$P$-graph} for short, is a topological small category
  $\Lambda$ as above endowed with a map, called the degree map,
  $d:\Lambda\to P$ which satisfies the following properties
  \begin{enumerate}
  \item the degree map $d:\Lambda\to P$ is continuous (where $P$ has
    the discrete topology);
  \item for all $(\mu,\nu)\in \Lambda^{(2)}$, $d(\mu\nu)=d(\mu)d(\nu)$
    and for all $v\in\Lambda^{(0)}$, $d(v)=e$;
  \item it has the unique factorization property: for all $m,n\in P$,
    the composition map $\Lambda^m *\Lambda^n\to \Lambda^{mn}$ is a
    homeomorphism.
  \end{enumerate}
\end{definition}

As a basic example of $P$-graph, we consider the graph of a semigroup
action. If $(X,P,T)$ is a locally compact semigroup action as in the
previous section, set
\begin{equation*}
  \Lambda^{(0)}=X,\quad \Lambda=X*P,\quad r(x,n)=x,\quad
  s(x,n)=x\cdot n\quad\text{and}\quad d(x,n)=n.
\end{equation*}
Composition is necessarily given by $(x,m)(x\cdot m,n)=(x,mn)$. It
results from our axioms of a semigroup action that $\Lambda$ is a
$P$-graph which we call the \emph{graph of the action}. For example,
if $(X,T)$ is a singly generated dynamical system as defined in the
previous section, it is useful to think of $(x,n)\in \Lambda=X*\N$ as
a finite path $(x, Tx, \ldots, T^{n-1}x)$ when $n\ge 1$. More
generally, we think of $(x,n)\in X*P$ as a finite path $\set{(x\cdot
  m):m\le n}$, although this latter set need not be finite.

Of course, not all $P$-graphs arise as graphs of a semigroup
action. The topological $\N$-graphs are exactly the usual topological
graphs. A topological graph is given by a pair of locally compact
Hausdorff spaces $(E,V)$ with two maps $r,s: E\to V$, $r$ continuous
and $s$ local homeomorphism. The space $\Lambda$ of finite paths is
the disjoint union over $\N$ of the spaces $E^{(n)}$ of paths of
length $n$, where $E^{(0)}=V$, $E^{(1)}=E$ and
\begin{equation*}
  E^{(n)}=\set{e_1e_2\ldots e_n: \text{$e_i\in E$, $s(e_i)=r(e_{i+1})$
      for $i=1,\ldots, n-1$}}.
\end{equation*}
endowed with the product topology, for $n\ge 2$. It is a topological
category with $\Lambda^{(0)}=V$, the obvious range and source maps and
composition given by concatenation. It has an obvious degree map
$d:\Lambda\to\N$ where $d(\lambda)=n$ if and only if $\lambda\in
E^{(n)}$. This definition includes the graphs which appear in the
theory of graph C*-algebras, which is the case when $E$ and $V$ are
discrete spaces. A singly generated dynamical system $(X,T)$ can be
viewed as a topological graph with $V=X$, $E=\operatorname{dom}(T)$,
$r(x)=x$ and $s(x)=Tx$. Its space of finite paths $\Lambda$ agrees
with $X*\N$. As another example of topological graph, consider
$(E=\T,V=\T)$ where $\T$ is the circle $|z|=1$ and the range and
source maps are respectively $z\mapsto z^2$ and $z\mapsto
z^3$. Topological graphs where the range and source maps are both
local homeomorphisms are called polymorphisms in
\cite{arzren:oaqft97}.

By analogy with the above examples, the elements of a higher-rank
graph $\Lambda$ are called finite paths. We define $\mu\le\lambda$ if
there exists $\nu$ such that $\lambda=\mu\nu$. This is a pre-order
relation which shares some of the properties of the pre-order relation
we have defined on the semigroup $P$ in the previous section. In
particular, suppose that $\lambda$ and $\mu$ have a common upper bound
$\nu$. Then $d(\lambda)$ and $d(\mu)$ have $d(\nu)$ as a common upper
bound. If $P$ is quasi-lattice ordered, $d(\lambda)$ and $d(\mu)$ have
a least common upper bound $p$. Therefore, there exists $\nu'\le\nu$
c.u.b.\ of $\lambda$ and $\mu$ with $d(\nu')= p$. We say that $\nu'$
is a l.u.b.\ of $\lambda$ and $\mu$. Such a l.u.b.\ need not be
unique. Given $A,B$ subsets of $\Lambda$, we denote by $A\vee B$ the
set of elements which are l.u.b.\ of some pair $(\lambda,\mu)\in
A\times B$. Given $\mu\le\lambda$, we define the segment
$[\mu,\lambda]:=\{\nu: \mu\le\nu\le\lambda\}$. We shall use the
following notation. If $\lambda\in\Lambda$ and $n\in P$ are such that
$n\le d(\lambda)$, then $\lambda$ can be written uniquely
$\lambda=\mu\nu$ where $d(\mu)=n$. Then we define $\lambda\ndot
n:=\nu$. Conversely, given $\mu\in\Lambda$, we can define $\mu\nu$ for
all $\nu\in r^{-1}(s(\mu))$.

We shall need some further assumptions on our higher-rank graphs.
\begin{definition} One says that the $P$-graph $\Lambda$ is
  \begin{enumerate}
  \item \emph{$(r,d)$-proper} if the map $(r,d):\Lambda\to
    \Lambda^{(0)}\times P$ is proper;
  \item \emph{compactly aligned} if $P$ is quasi-lattice ordered and for all
    compact subsets $A, B \subset\Lambda$, the subset $A\vee B$ is
    compact.
  \end{enumerate}
\end{definition}

In the setting of a discrete $\N$-graph, condition (a) means that a
vertex emits finitely many edges while condition (b) is always
satisfied. The graph of the action of a semigroup always satisfies (a)
since $(r,d)$ is the injection of $\Lambda=X*P$ into $X\times P$. When
$P$ is quasi-lattice ordered, condition (a) implies condition
(b). Indeed, if $\nu$ belongs to $A\vee B$ where $A,B$ are subsets of
$\Lambda$, then $r(\nu)$ belongs to $r(A)\cap r(B)$ and $d(\nu)$
belongs to $d(A)\vee d(B)$. If $A$ and $B$ are compact, $r(A)\cap
r(B)$ is compact and $d(A)\vee d(B)$ is finite.

Let $\Lambda$ be a $P$-graph. Set
\begin{equation*}
  \Lambda*P=\set{(\lambda,m)\in \Lambda\times P: m\le d(\lambda)}
\end{equation*}
and define
\begin{equation*}
  T: \Lambda*P\to \Lambda
\end{equation*}
by $T(\lambda,m):=\lambda\ndot m=\nu$ if $d(\lambda)=mn$ and
$\lambda=\mu\nu$ with $d(\mu)=m$ and $d(\nu)=n$.

\begin{prop} Let $\Lambda$ be a $P$-graph. Define $T$ as above.
  % \begin{enumerate}
  % \item $T$ is an action of $P$ on $\Lambda$ by partial local
  %   homeomorphisms;
  % \item if $P$ is quasi-lattice ordered, this action is well-directed.
  % \end{enumerate}
Then $T$ is a directed action of $P$ on $\Lambda$ by partial local
homeomorphisms.
\end{prop}

\begin{proof} The domain of $T_m$ is the open set
  $U(m)=\{\lambda\in\Lambda: m\le d(\lambda)\}$.  Its range
  \begin{equation*}
    V(m)=\set{\nu\in\Lambda: \text{there exists $\mu\in \Lambda^m$
        such that $ (\mu,\nu)\in
        \Lambda^{(2)}$}}=r^{-1}(s(\Lambda^m))
  \end{equation*}
  is also open because $s$ is open. Let us show that $T_m$ is
  continuous and open. Since $U(m)$ is the union of the open subsets
  $\Lambda^{ma}$ when $a$ runs over $P$, it suffices to study its
  restriction to $\Lambda^{ma}$. This restriction factors as
$$\Lambda^{ma}\to \Lambda^m*\Lambda^a\to \Lambda$$
The first map is a homeomorphism by assumption. The second map is the
restriction of the projection onto the second factor, and is open
because $s$ is open. Therefore the composition is continuous and open.
Moreover $s:\Lambda\to\Lambda^{(0)}$ is a local homeomorphism. On an
open subset $U$ of $\Lambda$ on which $s$ is injective,
$${T_m}_{|U}: U(m)\cap U\to \Lambda$$ is a homeomorphism onto an open subset.

To see that the action is directed, suppose that
 $U(m)\cap U(n)\not=\emptyset$. Then $m,n\le
d(\lambda)$ for some $\lambda\in U(m)\cap U(n)$.  Then $U(m)\cap U(n)
\subset U(d(\lambda))$.
\end{proof}

If $P$ is a subsemigroup of a discrete group $Q$
and $\Lambda$ is a $P$-graph, then we can form the semi-direct product
groupoid $G(\Lambda, P, T)$. This groupoid is proper, which means that
the map $(r,s): G(\Lambda, P, T)\to \Lambda\times\Lambda$ is proper.
The fine structure of such groupoids can be interesting (in the group
case, see for example \cites{echeme:em11,echwil:tams14}).  In the
theory of graph algebras, $\Lambda$ is the space of finite paths. It
is fruitful (and necessary) to construct a larger space, which is
called the path space and which includes infinite paths. This is what
we do in the next subsection.
% Therefore, the action of $P$ on $\Lambda$ does not produce an
% interesting groupoid.

\subsection{The path space \boldmath $\Omega$}
Just as in the case of a graph, we want to define a space of paths,
both finite and infinite. Our construction is directly inspired by a
construction of A. Nica in \cite{nic:jot92} which we recall
below. In
fact, our construction agrees with Nica's when $\Lambda=X*P$ and when
$X$ is reduced to one point. The same idea of defining paths as
hereditary directed subsets of the graph appears also in
\cite{bsv:ijm13}*{Section 3}.

Let us first summarize the exposition given in \cite{nic:jot92}. There
$P$ is a subsemigroup of a discrete group $Q$ and it is assumed to be
quasi-lattice ordered. One embeds $P$ into the space $\sset{0,1}^P$ of
all subsets of $P$ endowed with the product topology by sending $m\in
P$ to the segment $j(m)=[e,m]$. The Wiener-Hopf closure of $P$ is the
closure of $j(P)$ in $\{0,1\}^P$. It is denoted by $\Omega(P)$. Nica
remarks that the elements of $\Omega(P)$ are exactly the non-empty
hereditary and directed subsets of $P$. As we will see below, a
similar construction can be employed to define a closure $\Omega$ of a
topological $P$-graph~$\Lambda$.

Recall that we have defined on $\Lambda$ the pre-order $\mu\le\lambda$
if there exists $\nu$ such that $\lambda=\mu\nu$. Note that this
implies that $r(\mu)=r(\lambda)$ and $d(\mu)\le d(\lambda)$.

Following \cite{nic:jot92}, we call a subset $A$ of $\Lambda$
\emph{hereditary} if $\mu\le\lambda\in A$ implies $\mu\in A$, and
\emph{directed} if any two elements of $A$ have a c.u.b.\ in
$A$. Hereditary and directed subsets are called filters in
\cite{exe:sf09,bsv:ijm13}. Our 
terminology is the same as in \cite{spi:tams14}. We
denote by $\tilde\Omega$ the set of all hereditary and directed closed
subsets of $\Lambda$ and by $\Omega$ the set of all non-empty
hereditary and directed closed subsets of $\Lambda$. We view
$\tilde\Omega$ and $\Omega$ as subsets of the space ${\mathcal
  C}(\Lambda)$ of all closed subsets of $\Lambda$ endowed with the
Fell topology \cite{wil:crossed}*{\S H.1}.  Recall that a basis for
the Fell topology on $\mathcal C(\Lambda)$ is given by sets of the
form
\begin{equation*}
  \mathcal U(K;U_{1},\dots,U_{m})=\set{F\in \mathcal
    C(\Lambda):\text{$F\cap K=\emptyset$ and $F\cap U_{i}\not=\emptyset$}}
\end{equation*}
where $K\subset \Lambda$ is compact and each $U_{i}\subset \Lambda$ is
open. Then as in
\cite{wil:crossed}*{Lemma~H.2}\footnote{Unfortunately, the statement
  of \cite{wil:crossed}*{Lemma~H.2} omits the observation that (F1)
  and (F2) must hold for every subnet of the original net.} a net
$(F_{\beta})$ converges to $F$ in $\mathcal{C}(\Lambda)$ if and only
if every subnet $(F_{i})$ of $(F_{\beta})$ is such that
\begin{enumerate}[\qquad (F1)]
\item given $\lambda_{i}\in F_{i}$ such that $\lambda_{i}\to \lambda$,
  then $\lambda\in F$, and
\item if $\lambda\in F$, then there is a subnet $(F_{i_{j}})$ and
  $\lambda_{j}\in F_{i_{j}}$ such that $\lambda_{j}\to \lambda$.
\end{enumerate}

\begin{lemma}\label{rangemap}%
  \begin{enumerate}
  \item Let $A$ be a non-empty directed subset of $\Lambda$. Then $A$
    is contained in $x\Lambda:=r^{-1}(x)$ for some (necessarily
    unique) $x\in\Lambda^{(0)}$, which will be written $r(A)$;
  \item the map $r:\Omega\to\Lambda^{(0)}$ is continuous.
  \end{enumerate}
\end{lemma}

\begin{proof}
  (a) By definition, the relation $\mu\le\lambda$ implies that
  $r(\mu)=r(\lambda)$.  Let $\mu,\nu\in A$. Since there exists
  $\lambda$ such that $\mu,\nu\le\lambda$, we must have
  $r(\mu)=r(\nu)$.

  (b) Suppose that $A_\alpha, A\in\Omega$ and $A_\alpha$ tends to
  $A$. Let $x_\alpha=r(A_\alpha)$ and $x=r(A)$. Let $U$ be an open
  neighborhood of $x$. Since $A\cap r^{-1}(U)\not=\emptyset$, there
  exists $\alpha_0$ such that $A_\alpha\cap r^{-1}(U)\not=\emptyset$
  for all $\alpha\ge \alpha_0$. Then, $x_\alpha\in U$.
\end{proof}

\begin{lemma}\label{closedrelation} Assume $(r,d)$ is proper and that
  $P$ is contained in a group $Q$. If $\lambda_i$ converges to
  $\lambda$, $\mu_i$ converges to $\mu$ and for all $i$,
  $\mu_i\le\lambda_i$, then $\mu\le\lambda$.
\end{lemma}

\begin{proof}
  There exists a net $\nu_i$ such that $\lambda_i=\mu_i\nu_i$. Since
  $d(\lambda_i)=d(\mu_i)d(\nu_i)$, $d(\nu_i)$ is eventually
  constant. Since $r(\nu_i)=s(\mu_i)$, it is contained in some compact
  subset of $\Lambda^{(0)}$. Because of $(r,d)$-properness, there is a
  subnet $(\nu_i)$ converging to some $\nu$. Then $\lambda=\mu\nu$.
\end{proof}

\begin{lemma} Let $A$ be a subset of $\Lambda$.
  \begin{enumerate}
  \item If $A$ is directed (resp., hereditary), then $d(A)$ is
    directed (resp., hereditary).
  \item If $A$ is directed, the restriction to $A$ of the degree map
    $d_{|A}:A\to P$ is a bijection onto $d(A)$.
  \end{enumerate}
\end{lemma}

\begin{proof}
  Let $m,n\in d(A)$. There exist $\mu,\nu\in A$ such that $m=d(\mu)$
  and $n=d(\nu)$. If $A$ is directed, there exists $\lambda\in A$ such
  that $\mu,\nu\le\lambda$. Then $m,n\le d(\lambda)$. Therefore $d(A)$
  is directed. Suppose that $m\le n$ and that $n=d(\lambda)$ with
  $\lambda\in A$. We write $n=mp$. By unique factorization, we can
  write $\lambda=\mu\pi$, where $d(\mu)=m$ and $d(\pi)=p$. If $A$ is
  hereditary, then $\mu\in A$ and $m\in d(A)$. Therefore, $d(A)$ is
  hereditary.

  Suppose that $A$ is directed. Let $\mu, \nu\in A$ such that
  $d(\mu)=d(\nu)$. Let $\lambda$ be a c.u.b.\ of $(\mu, \nu)$. By
  unique factorization of $\lambda$, we have the equality $\mu=\nu$.
\end{proof}

\begin{lemma}
  \label{lem-An} Let $\Lambda$ be a $P$-graph.
  \begin{enumerate}
  \item If $A\subset\Lambda$ is hereditary (resp., directed) and $n\in
    P$, the subset
    \begin{equation*}
      A\ndot n=\set{\nu \in\Lambda:\text{there exists $\mu\in \Lambda^n$
          such that $ \mu\nu\in A$}}
    \end{equation*}
    is hereditary (resp., directed). If $(r,d)$ is proper and $A$ is
    directed and closed, then $A\ndot n$ is directed and closed.

  \item Assume $(r,d)$ is proper and that $P$ is a subsemigroup of a
    group $Q$. If $B\subset\Lambda$ is directed, hereditary and closed
    and if $\mu\in r(B)\Lambda$, then
    \begin{equation*}
      \mu B=\bigcup_{\nu\in B}\set{\lambda\in\Lambda:\lambda\le\mu\nu}
    \end{equation*}
    is directed, hereditary and closed. Moreover, if $A\ndot n$ is
    non-empty, there is a unique $\mu\in\Lambda^n\cap A$ such that
    $A=\mu(A\ndot n)$.
  \end{enumerate}
\end{lemma}

\begin{proof}
  (a) Suppose that $A$ is hereditary . Let $b\le (a\ndot n)$ where
  $a\in A$.  Then, $a=\mu(a\ndot n)$ where $d(\mu)=n$.  Moreover,
  $a\ndot n=b\ndot c$ for some $c\in P$. Thus, $a=\mu (b\ndot c)=(\mu
  b)\ndot c$, and $\mu b\le a$. Since $A$ is hereditary, $\mu b \in
  A$, hence $b\in A\ndot n$. We have shown that $A\ndot n$ is
  hereditary.

  Suppose that $A$ is directed. Consider $a\ndot n$ and $ b\ndot n$
  where $a,b\in A$. We want a c.u.b.\ for $a\ndot n, b\ndot n$ in
  $A\ndot n$. Let $c$ be a c.u.b.\ for $a, b$ in $A$. Then $c\ndot n$
  is defined and is a c.u.b.\ for $a\ndot n$ and $ b\ndot n$.

  Suppose that $A$ is directed and closed, that $a_i$ belongs to $A$
  and that $a_i \ndot n$ converges to some $b\in \Lambda$. We write
  $a_i=\mu_i(a_i\ndot n)$, where $d(\mu_i)=n$. Since we also have
  $r(\mu_i)=r(A)$, by $(r,d)$-properness, there is a subnet $(\mu_j)$
  converging to some $\mu$ such that $r(\mu)=r(a), s(\mu)=r(b)$ and
  $d(\mu)=n$ Then $a_j$ converges to $a=\mu (b)$. Since $A$ is closed,
  $a$ belongs to $A$ and $b=a\ndot n$ belongs to $A\ndot n$.

  (b) To see that $\mu B$ is directed, suppose that $\lambda\in B$ is
  a c.u.b.\ for $\nu_1,\nu_2\in B$. Then $\mu\lambda$ is a c.u.b.\ for
  $\mu\nu_1,\mu\nu_2$. The set $\mu B$ is hereditary by
  construction. To see that it is closed, consider a net
  $\lambda_\alpha$ in $\mu B$ converging to $\lambda$. We distinguish
  two cases: if $d(\lambda)\le d(\mu)$, then $d(\lambda_\alpha)\le
  d(\mu)$ for $\alpha$ large enough. This implies
  $\lambda_\alpha\le\mu$, hence $\lambda\le\mu$. If $d(\lambda)\le
  d(\mu)$ does not hold, there is a subnet $\lambda_\beta$ for which
  $d(\lambda_\beta)\le d(\mu)$ does not hold. Then, we can write
  $\lambda_\beta=\mu\nu_\beta$ with $\nu_\beta\in B$. Since
  $d(\lambda_\beta)= d(\lambda)$ for $\beta$ large enough, this fixes
  $d(\nu_\beta)$ (if $P\subset Q$). We also have
  $r(\nu_\beta)=s(\mu)$. By $(r,d)$-properness, there is a converging
  subnet $\nu_\gamma$. Its limit $\nu$ belongs to $B$ because $B$ is
  closed. We have $\lambda=\mu\nu$, hence $\lambda$ is in $\mu B$.

  If $A\ndot n$ is non-empty, there is $\lambda\in A$ such that
  $d(\lambda)\ge n$. We write $\lambda=\mu\nu$ with $\mu\in
  \Lambda^n$. Since $\mu\le\lambda$, $\mu$ belongs to $A$. Let
  $\lambda'$ be another element of $A$ such that $d(\lambda')\ge n$. We
  can write $\lambda'=\mu'\nu'$. The unique factorization of a c.u.b.\
  of $(\lambda,\lambda')$ gives $\mu=\mu'$. Let us compare $A$ and
  $\mu(A\ndot n)$. If $\lambda\in A$, the existence of a c.u.b.\ for
  $(\lambda,\mu)$ shows that $\lambda\in\mu(A\ndot n)$. Conversely,
  suppose that $\lambda\le\mu\nu$ for $\nu\in A\cdot n$. There exists
  $\mu'\in\Lambda^n$ such that $\mu'\nu\in A$. By the above,
  $\mu'=\mu$, therefore $\mu\nu\in A$, hence $\lambda\in A$.
\end{proof}

\begin{prop} Let $\Lambda$ be a $P$-graph, where $P$ is a
  quasi-lattice ordered subsemigroup of a group $Q$ and $\Lambda$ is
  $(r,d)$-proper.
  \begin{enumerate}
  \item the set $\tilde\Omega$ of all hereditary and directed closed
    subsets of $\Lambda$ is a closed subset of the space ${\mathcal
      C}(\Lambda)$ of closed subsets of $\Lambda$ equipped with the
    Fell topology.
  \item for all $\lambda\in\Lambda$, $F(\lambda):=\{\mu\in\Lambda:
    \mu\le\lambda\}$ belongs to $\tilde\Omega$;
  \item $F(\Lambda)$ is dense in $\tilde\Omega$;
  \item if $n\le d(\lambda)$, $F(\lambda)\ndot n=F(\lambda \ndot n)$;
    if $s(\mu)=r(\lambda)$, $\mu F(\lambda)=F(\mu\lambda)$;
  \item the map $F:\Lambda\to {\mathcal C}(\Lambda)$ is injective and
    continuous.
  \end{enumerate}
\end{prop}

\begin{proof} (a) Suppose that $A_\alpha$ converges to $A$ in
  ${\mathcal C}(\Lambda)$ and that the $A_\alpha$'s are hereditary and
  directed. Let us show that $A$ is hereditary. Let $\lambda\in A$ and
  $\mu\le\lambda$. Suppose that $\mu\notin A$. Since $A$ is closed,
  there is an open set $U$ and a compact set $K$ such that $\mu\in
  U\subset K\subset \Lambda\setminus A$. The set
  $U\Lambda=\set{\alpha\beta:\text{$\alpha\in U$ and
      $\beta\in\Lambda$}}$ is open (because multiplication is assumed
  to open and continuous), and contains $\lambda$. Thus we have $A\cap
  K=\emptyset$ and $A\cap UP\not=\emptyset$. There is $\alpha_0$ such
  that for $\alpha\ge\alpha_0$, $A_\alpha\cap K=\emptyset$ and
  $A_\alpha\cap U\Lambda\not=\emptyset$. If $\lambda_\alpha$ belongs
  to $A_\alpha\cap U\Lambda$, there exists $\mu_\alpha\in U$ such that
  $\mu_\alpha\le\lambda_\alpha$. Since $A_\alpha$ is hereditary,
  $\mu_\alpha$ belongs to $A_\alpha$. This contradicts $A_\alpha\cap
  K=\emptyset$.

  Let us show that $A$ is directed. Let $\mu, \nu$ be in $A$. There
  exist nets $\mu_\alpha$ and $\nu_\alpha$ in $A_\alpha$ converging
  respectively to $\mu$ and $\nu$. Let $\lambda_\alpha$ be a l.u.b.\
  of $(\mu_\alpha,\nu_\alpha)$ belonging to $A_\alpha$. (Such a
  l.u.b.\ exists because $A_\alpha$ is directed and hereditary and $P$
  is quasi-lattice ordered). Since $P$ is quasi-lattice ordered
    and $(r,d)$ is proper, $\Lambda$ is compactly aligned. Hence
  the net $\lambda_\alpha$ has a convergent subnet. Let $\lambda$ be
  its limit. Since $A_\alpha$ converges to $A$, $\lambda$ is in $A$
  (by (F1)). Since $\mu_\alpha,\nu_\alpha\le \lambda_\alpha$,
  Lemma~\ref{closedrelation} gives $\mu,\nu\le \lambda$. This shows
  that $A$ is directed.

  (b) The set $F(\lambda)$ is obviously hereditary and
  directed. According to Lemma~\ref{closedrelation}, it is closed.
  
  (c) Let $A$ be a hereditary and directed closed subset of
  $\Lambda$. It will suffice to see that the net
  $(F(\lambda))_{\lambda\in A}$ converges to $A$.  Let $K$ be a
  compact set in $\Lambda$ and $U_1,\ldots, U_n$ such that $A\cap
  K=\emptyset$ and $A\cap U_i\not=\emptyset$. Pick $\lambda_i\in A\cap
  U_i$ and let $\underline\lambda$ be a c.u.b.\ of $\lambda_1,\ldots,
  \lambda_n$ in $A$. If $\lambda\ge\underline\lambda$, $F(\lambda)\cap
  U_i\not=\emptyset$ because it contains $\lambda_i$. Since
  $F(\lambda)$ is contained in $A$, its intersection with $K$ is
  empty.

  (d) Straightforward.

  (e) Since we assume that $P\cap P^{-1}=\{e\}$, the relation $\le$ is
  an order relation, hence the injectivity of $F$. Suppose that
  $\lambda_\alpha\to\lambda$. Let $K$ be a compact subset of $\Lambda$
  and let $U_1,\ldots, U_n$ be
  open subsets of $\Lambda$ such that $F(\lambda)\cap K=\emptyset$ and
  $F(\lambda)\cap U_i\not=\emptyset$. If there is no $\alpha_0$ such
  that $F(\lambda_\alpha)\cap K=\emptyset$ for all
  $\alpha\ge\alpha_0$, there are subnets $\mu_\beta\le\lambda_\beta$
  with $\mu_\beta\to\mu$ in $K$. This is not possible since we have
  then $\mu\le\lambda$.  By assumption, $\lambda$ belongs to the open
  set $U_1\Lambda\cap\ldots\cap U_n\Lambda$, therefore there exists
  $\alpha_1\ge\alpha_0$ such that for all $\alpha\ge\alpha_1$,
  $\lambda_\alpha\in U_1\Lambda\cap\ldots\cap U_n\Lambda$.  Hence we
  eventually have $F(\lambda_{\alpha})\cap U_{i}\not=\emptyset$.
\end{proof}

We shall see later that the map $F:\Lambda\to {\mathcal C}(\Lambda)$
is not necessarily a homeomorphism onto its image
(Remark~\ref{rem-last}(c)).

Recall that at the beginning of the subsection we defined
$\Omega=\tilde\Omega\setminus\{\emptyset\}$.  We now define
\begin{align*}
  \Omega*P&=\set{(A,n)\in\Omega\times P: A\ndot n\not=\emptyset} \\
  &= \set{(A,n)\in \Omega\times P: \text{there exists $\lambda\in A$
      such that $d(\lambda)\ge n$}},
\end{align*}
and $T:\Omega*P\to \Omega$ sending $(A,n)$ to $A\ndot n$. As before,
we define
\begin{equation*}
  U(n)=\set{A\in \Omega: (A,n)\in \Omega*P}\quad\text{and}\quad
  V(n)=\set{A\ndot n: (A,n)\in \Omega*P}
\end{equation*}
and we denote by $T_n: U(n)\to V(n)$ the map sending $A$ to $A\ndot
n$.

\begin{definition} We define a \emph{path} in $\Lambda$ as a non-empty
  hereditary closed subset of $\Lambda$. The space $\Omega$ is called
  the \emph{path space} of $\Lambda$. The above map $T$ is called the
  \emph{shift} on the path space.
\end{definition}

\begin{thm} Let $\Lambda$ be a $P$-graph, where $P$ is quasi-lattice
  ordered and $\Lambda$ is $(r,d)$-proper. Let $(\Omega, P, T)$ be as
  above. Then,
  % \begin{enumerate}
  % \item $T$ is an action of $P$ on $\Omega$ by partial local
  %   homeomorphisms and
  % \item this action is well-directed.
  % \end{enumerate}
$T$ is a directed action of $P$ on $\Omega$ by partial local
homeomorphisms. 
\end{thm}

\begin{proof}
  First, we observe that $U(n)$ and $V(n)$ are open
  subsets. Indeed, we have
  \begin{equation*}
    U(n)=\{A\in\Omega: A\cap\bigcup\nolimits_{m\ge
      n}\Lambda^m\not=\emptyset\}\quad\hbox{and}\quad
    V(n)=r^{-1}(s(\Lambda^n)),
  \end{equation*}
  where $r:\Omega\to\Lambda^{(0)}$ is the range map defined in Lemma
  \ref{rangemap}. Let us show that $T_n:U(n)\to V(n)$ is a local
  homeomorphism. To show that $T_n$ is continuous, suppose that
  $A_\alpha$ converges to $A$. We need to show that $A_\alpha \ndot n$
  converges to $A\ndot n$. It suffices to show that every subnet
  $(A_\beta \ndot n)$ satisfies (F1) and (F2). Suppose that
  $\nu_\beta\in A_\beta \ndot n$ converges to $\nu$. Then there is
  $\mu_\beta\in\Lambda^n$ such that $\mu_\beta\nu_\beta\in
  A_\beta$. Then $r(\mu_\beta)=r(A_\beta)$ converges to $r(A)$. By
  $(r,d)$-properness, there is a subnet $\mu_\gamma\in \Lambda^n$
  which converges to some $\mu\in\Lambda^n$. Then $\mu\nu$ belongs to
  $A$ and $\nu$ belongs to $A\ndot n$. Condition (F2) is clear: let
  $\nu\in A\ndot n$. There exists $\mu\in\Lambda^n$ such that
  $\mu\nu\in A$. There is a subnet $A_\beta$ and $\lambda_\beta\in
  A_\beta$ such that $\lambda_\beta$ converges to $\mu\nu$. Then
  $\lambda_\beta \ndot n\in A_\beta \ndot n$ converges to $\nu$.

  Let $U$ be an open subset of $\Lambda^n$ such that $s_{|U}:U\to
  s(U)$ is a homeomorphism; we denote by $\sigma$ the inverse of
  $s_{|U}$. Then $\tilde U=\set{A\in U(n): A\cap U\not=\emptyset}$ is
  open. Each $A\in \tilde U$ contains a unique $\mu\in U$ which is
  given by $\mu=\sigma(r(A\ndot n))$. Thus, the restriction of $T_n$
  to $\tilde U$ is a bijection of $\tilde U$ onto $T_n(\tilde
  U)=r^{-1}(s(U))$ having as inverse map $B\mapsto \sigma\circ
  r(B)B$. To show that this inverse map is continuous, it suffices to
  show that the product map $\Lambda^n*\Omega\to\Omega$ sending $(\mu,
  B)$ to $\mu B$ is continuous. Consider a net $(\mu_\alpha,
  B_\alpha)$ converging to $(\mu, B)$. We will show that $\mu_\alpha
  B_\alpha$ converges to $\mu B$ by checking that every subnet
  $\mu_\beta B_\beta$ satisfies (F1) and (F2). For (F1), we proceed as
  in the proof of Lemma~\ref{lem-An}(b). Consider a net
  $\lambda_\beta\in\mu_\beta B_\beta$ converging to $\lambda$. We
  distinguish two cases: if $d(\lambda)\le d(\mu)$, then
  $d(\lambda_\beta)\le d(\mu_\beta)$ for $\beta$ large enough. This
  implies $\lambda_\beta\le\mu_\beta$, hence $\lambda\le\mu$. If
  $d(\lambda)\le d(\mu)$ does not hold, there is a subnet
  $\lambda_\gamma$ for which $d(\lambda_\gamma)\le d(\mu_\gamma)$ does
  not hold. Then, we can write $\lambda_\gamma=\mu_\gamma\nu_\gamma$
  with $\nu_\gamma\in B_\gamma$. Since $d(\lambda_\gamma)= d(\lambda)$
  and $d(\mu_\gamma)=d(\mu)$ for $\gamma$ large enough, this fixes
  $d(\nu_\gamma)$. We also have $r(\nu_\gamma)=s(\mu_\gamma)$. By
  $(r,d)$-properness, there is a converging subnet $\nu_\delta$. Its
  limit $\nu$ belongs to $B$ because $B$ is closed. We have
  $\lambda=\mu\nu$, hence $\lambda$ is in $\mu B$. Let us check
  (F2). Suppose that $\lambda$ belongs to $\mu B$. Suppose first that
  $d(\lambda)\le d(\mu)$. We have $n=pq$ with $p=d(\lambda)$. For
  $\beta$ large enough, $d(\mu_\beta)=n$ and we have a unique
  factorization $\mu_\beta=\lambda_\beta\rho_\beta$ where
  $d(\lambda_\beta)=p$. By $(r,d)$-properness, $\lambda_\beta$ has a
  subnet converging to some $\lambda'$; since $\rho_\beta=\mu_\beta
  \ndot p$ converges to $\mu \ndot p$, we have $\mu=\lambda'(\mu \ndot
  p)$; by unique factorization, $\lambda'=\lambda$. Therefore
  $\lambda_\beta$ converges to $\lambda$. Since $\lambda_\beta$
  belongs to $\mu_\beta B_\beta$, we are done if $d(\lambda)\le
  d(\mu)$. Suppose now that $\lambda=\mu\nu$ where $\nu\in B$. There
  exists a subnet $B_\gamma$ and $\nu_\gamma\in B_\gamma$ converging
  to $\nu$. Then $\mu_\gamma\nu_\gamma$ belongs to $\mu_\gamma
  B_\gamma$ and converges to $\mu\nu$; therefore, 
  $(\mu,B)\mapsto \mu B$ is continuous.

  To see that the action is directed, consider
  \begin{equation*}
    U(n)=\set{A\in \Omega: \text{there exists $ \lambda\in A$ such that
        $ d(\lambda)\ge n$}}. 
  \end{equation*}
  Assume that $U(m)\cap U(n)\not=\emptyset$ and let $A\in U(m)\cap
  U(n)$. There exist $\mu,\nu\in A$ such that $d(\mu)\ge m$ and
  $d(\nu)\ge n$. Since $A$ is directed, there exists $\lambda\in A$
  greater than $\mu$ and $\nu$. Then $d(\lambda)$ is greater than $m$
  and $n$. % In fact, the action is well-directed: we have
  % $d(\lambda)\ge m\vee n$. By the unique factorization property, we
  % can find $\lambda'$ such that $\lambda\ge\lambda'\ge \mu,\nu$ and
  % $d(\lambda')=m\vee n$. Then $\lambda'$ also belongs to $A$.
\end{proof}

Thus, under our assumptions on $\Lambda$ and $P$, we can construct the
semi-direct product groupoid $G(\Omega, P, T)$ according to
Proposition \ref{prop-semidirect-product-groupoid}.

\begin{definition} Let $\Lambda$ be a $P$-graph, where $P$ is
  quasi-lattice ordered and $\Lambda$ is $(r,d)$-proper and let $T$ be
  the action of $P$ on the path space $\Omega$. The groupoid
  $G(\Omega, P, T)$ is called the \emph{Toeplitz groupoid} of the topological
  higher rank graph $\Lambda$. Its C*-algebra is called the \emph{Toeplitz
  algebra} of $\Lambda$ and denoted by $C^*(\Lambda)$.
\end{definition}

This construction is the same as in the work of T. Yeend, for example
\cite{yee:jot07}. The main differences are that we consider an
arbitrary quasi-lattice ordered semigroup $P$ rather than $\N^d$ and
that we make an explicit (rather than implicit) use of the Fell
topology on a space of closed subsets to define the topological path
space $\Omega$. In \cite{yee:jot07}, the path space $\Omega$ is
denoted by $X_\Lambda$ and the groupoid $G(\Omega, P, T)$, called the
path groupoid, is denoted by $G_\Lambda$.

\subsection{The boundary path space \boldmath $\partial\Omega$}

We continue to assume that $P$ is quasi-lattice ordered and that
$\Lambda$ is $(r,d)$-proper.  In particular, $\Lambda$ is compactly
aligned. 
The Cuntz-Krieger algebra of the $P$-graph $\Lambda$ is described in
\cite{yee:jot07} as the C*-algebra of the reduction of $G(\Omega, P,
T)$ to a closed invariant subset $\partial\Omega$ called the boundary
path space. Let us describe the boundary path space in our
presentation. Recall that the elements of $\Omega$ are the non-empty
closed hereditary and directed subsets of $\Lambda$.

\begin{definition}
  Let $\Lambda$ be a $P$-graph. We say that $E\subset\Lambda$ is
  \emph{exhaustive} if for all $\lambda\in\Lambda$ such that $r(\lambda)\in
  r(E)$, there exists $\mu\in E$ such that $(\lambda,\mu)$ has a
  c.u.b.
\end{definition}

\begin{definition} Let $\Lambda$ be a $P$-graph.
  \begin{enumerate}
  \item Given $A\in\Omega$, we say that $\lambda\in A$ is \emph{extendable
    in $A$} if for all $E\subset\Lambda$ which are exhaustive, compact
    and such that $r(E)$ is a neighborhood of $s(\lambda)$, there
    exists $\mu\in E$ such that $\lambda\mu\in A$.
  \item We say that $A\in \Omega$ is a \emph{boundary path} if all its
    elements are extendable in $A$.
  \end{enumerate}
  We define the \emph{boundary path space} $\partial\Omega$ as the subspace
  of all boundary paths.
\end{definition}

\begin{example} [Singly Generated Systems] Recall that this means a
  local homeomorphism $T:U\to V$, where $U,V$ are open subsets of a
  locally compact Hausdorff space $X$ and that
  \begin{equation*}
    \Lambda=X*\N=\set{(x,n)\in X\times \N: x\in U(n):=\operatorname{dom}(T^n)}
  \end{equation*}
  The non-empty closed hereditary directed subsets of $\Lambda$ are:
  \begin{equation*}
    F(x,n)=\set{(x,m): m\le n}
  \end{equation*}
  where $n\in\overline\N:=\N\cup\{\infty\}$, $x\in U(n)$ if $n$ is
  finite and $x\in U(\infty)=\bigcap U(n)$ if $n=\infty$. The boundary
  paths are:
$$F(x, \tau(x))\quad \hbox{where}\quad \tau(x)=\sup\{n\in\N: x\in U(n)\}$$
Note that here the boundary paths are exactly the maximal paths.
\end{example}

\begin{lemma}\label{extendability-is-hereditary} Let $A\in\Omega$. If
  $\lambda'\in A$ is extendable in $A$, then every
  $\lambda\le\lambda'$ is extendable in $A$.
\end{lemma}

\begin{proof} Let $E\subset\Lambda$ be exhaustive, compact and such
  that $r(E)$ is a neighborhood of $s(\lambda)$. We write
  $\lambda'=\lambda\xi$ and we let $n=d(\xi)$. We choose $U$ open
  relatively compact neighborhood of $\xi$ contained in $\Lambda^n$
  and such that $s_{|\overline{U}}$ is injective and
  $r(\overline{U})\subset s(E)$. We define $E'$ to be the the set of
  elements $\mu'$ of $\Lambda$ of minimal degree for which there
  exist $\mu\in E$ and $\eta\in {\overline U}$ such that
  $\mu\le\eta\mu'$. 
  % \begin{multline*}
  %   E'=\{\,\mu'\in\Lambda: \text{there exists $(\mu,\eta)\in
  %       E\times\overline{U}$} \\ \text{such that $d(\mu')$ is minimal with $
  %       \mu\le\eta\mu'$}\,\} .
  % \end{multline*}
  One checks that $E'$ is closed. Since $r(\mu)\in r(\overline{U})$
  and $d(\mu)$ lies in a finite set by the compact alignment property,
  $E'$ is compact. We are going to show that $E'$ is exhaustive and
  that $r(E')=s(\overline{U})$. By construction, $r(E')\subset
  s(\overline{U})$. Consider $\nu\in\Lambda$ such that $r(\nu)\in
  s(\overline{U})$. Pick $\eta\in \overline{U}$ such that
  $s(\eta)=r(\nu)$. Then $r(\eta\nu)=r(\eta)\in r(\overline{U})\subset
  s(E)$. Since $E$ is exhaustive, there exists $\mu\in E$ such that
  $(\eta\nu,\mu)$ has a c.u.b. This means the existence of $\alpha,
  \beta\in\Lambda$ such that $\mu\beta=(\eta\nu)\alpha$. We have then
  $\mu\le\eta(\nu\alpha)$, hence the existence of $\mu'\in E'$ such
  that $\mu'\le \nu\alpha$. In particular, $r(\nu)\in r(E')$, which
  shows that $r(E')=s(\overline{U})$. We have found $\mu'\in E'$ such
  that $(\mu',\nu)$ has a c.u.b. This shows that $E'$ is
  exhaustive. Since $\lambda'$ is extendable in $A$, there exists
  $\mu'\in E'$ such that $\lambda'\mu'\in A$. By definition of $E'$,
  there is $(\mu,\eta)\in E\times\overline{U}$ such that
  $\mu\le\eta\mu'$. Since $\eta$ and $\xi$ both belong to
  $\overline{U}$ and have same source, $\eta=\xi$. Since
  $\lambda\mu\le\lambda\xi\mu'=\lambda'\mu'$, $\lambda\mu\in A$. This
  shows that $\lambda$ is extendable in $A$.
\end{proof}

\begin{prop}
  \label{properties-of-boundaries}
  Let $\partial\Omega$ be the boundary path space of a $P$-graph
  $\Lambda$, where $P$ is a quasi-lattice ordered subsemigroup
  of a group $Q$ and $\Lambda$ 
  is $(r,d)$-proper. Then
  \begin{enumerate}
  \item $\partial\Omega$ is a closed subset of $\Omega$.
  \item If $A\in\partial\Omega$ and $n\in P$ such that $A\ndot
    n\not=\emptyset$, then $A\ndot n\in\partial\Omega$.
  \item If $B\in\partial\Omega$ and $\rho\in r(B)\Lambda$, then $\rho
    B\in\partial\Omega$.
  \end{enumerate}
\end{prop}
\begin{proof}
  (a) Suppose that $A_\alpha\to A$ and $A_\alpha\in\partial\Omega$. If
  $A$ is not a boundary path, there exists $\lambda\in A$ and
  $E\subset \Lambda$ exhaustive, compact with $r(E)$ a neighborhood of
  $s(\lambda)$ such that $\lambda\mu\not\in A$ for all $\mu\in E$. Let
  $V$ be an open neighborhood of $s(\lambda)$ contained in $r(E)$. Let
  $U$ be an open relatively compact neighborhood of $\lambda$ such
  that $s(U)\subset V$ and $U\subset\Lambda^n$, where
  $n=d(\lambda)$. We have $A\cap U\not=\emptyset$ and
  $A\cap\overline{U}E=\emptyset$. Let us check this second assertion:
  if $\lambda'\mu\in A$, with $\lambda'\in \overline{U}$ and $\mu\in
  E$ , then $\lambda'\in A$. Since $d(\lambda')=d(\lambda)$,
  $\lambda=\lambda'$. This is a contradiction. There exists $\alpha_0$
  such that for all $\alpha\ge\alpha_0$, $A_\alpha\cap
  U\not=\emptyset$ and $A_\alpha\cap\overline{U}E=\emptyset$. Let
  $\lambda_\alpha\in A_\alpha\cap U$. Then $r(E)$ is a neighborhood of
  $s(\lambda_\alpha)$ and for all $\mu\in E$ such that
  $r(\mu)=s(\lambda_\alpha)$, $\lambda_\alpha\mu$ does not belong to
  $A$. This contradicts the fact that $A_\alpha$ is a boundary path.
  
  (b) Let $A$ be a boundary path and $n\in P$ such that $A\cdot n$ is
  non-empty. Let us show that $A\ndot n$ is a boundary path. Recall
  that $\nu\in A\ndot n$ if and only if there exists $\rho\in
  \Lambda^n$ such that $\rho\nu\in A$.  Let $\nu\in A\ndot n$ and
  $E\subset \Lambda$ exhaustive, compact with $r(E)$ a neighborhood of
  $s(\nu)$. There exists $\rho\in\Lambda^n$ such that
  $\lambda=\rho\nu\in A$. Since $A$ is a boundary path, there exists
  $\mu\in E$ such that $\lambda\mu\in A$. Therefore $\nu\mu\in A\ndot
  n$.

  (c) We first show that every $\lambda\in \rho B$ of the form
  $\lambda=\rho\nu$, where $\nu\in B$, satisfies the property. Indeed,
  let $E$ be a subset of $\Lambda$ which is exhaustive, compact and
  such that $r(E)$ is a neighborhood of $s(\lambda)$. Since
  $s(\lambda)=s(\nu)$ and $B$ is a boundary path, there exists $\mu\in
  E$ such that $\nu\mu\in B$. Then $\lambda\mu=\rho(\nu\mu)$ belongs
  to $\rho B$. We apply Lemma \ref{extendability-is-hereditary} to
  conclude that $\rho B\in \partial \Omega$.
\end{proof}

\begin{cor} Let $\Lambda$ be a $P$-graph, where $P$ is quasi-lattice
  ordered and $\Lambda$ is $(r,d)$-proper. Let $(\Omega, P, T)$ and
  $G(\Omega, P, T)$ be as above. Then the boundary path space
  $\partial\Omega$ is a closed invariant subspace of $\Omega$ with
  respect to $G(\Omega, P, T)$.
\end{cor}

\begin{proof} The equivalence relation on $\Omega$ induced by
  $G(\Omega, P, T)$ is precisely $A\sim B$ if and only if there exist
  $m,n\in P$ such that $A\ndot m$ and $B\ndot n$ are non-empty and
  equal. If $A$ is a boundary path, so is $A\ndot m=B\ndot n$ by
  Proposition~\ref{properties-of-boundaries}(b). Since $B=\rho(B\ndot
  n)$ for some $\rho$, $B$ is a boundary path by part~(c) of the same
  proposition. Therefore, the space of boundary paths is invariant. We
  have shown above that it is closed.
\end{proof}

\begin{definition} The reduction $G(\partial\Omega, P, T)$ of the
  Toeplitz groupoid $G(\Omega, P, T)$ is called the
  \emph{Cuntz-Krieger groupoid} of the topological higher rank graph
  $\Lambda$. Its \cs-algebra is called the \emph{Cuntz-Krieger
    algebra} of $\Lambda$ and denoted by $C^*(\partial\Lambda)$.
\end{definition}

Note that $G(\partial\Omega, P, T)$ is the semi-direct product
groupoid of the semigroup action $(\partial\Omega, P, T)$. Since the
action on $\Omega$ is directed, so is the action on $\partial\Omega$.
Hence Theorem~\ref{thm-app-directed-sg-actions} (and
\cite{anaren:amenable00}*{Corollary~6.2.14}) give us the following.

\begin{cor} Let $P$ a quasi-lattice ordered subsemigroup of a group
  $Q$ and $\Lambda$ be a $P$-graph which is is $(r,d)$-proper. If the
  group $Q$ is amenable, then the Toeplitz 
  groupoid $G(\Omega, P, T)$ and the Cuntz-Krieger groupoid
  $G(\partial\Omega, P, T)$ are amenable. Therefore the Toeplitz
  algebra $C^*(\Lambda)$ and the Cuntz-Krieger algebra
  $C^*(\partial\Lambda)$ are nuclear.
\end{cor}

\subsection{Topological higher rank graphs coming from semigroup
  actions}

We have seen that a semigroup action $(X,P,T)$ gives the topological
higher rank graph $\Lambda=X*P$. If $P$ is quasi-lattice ordered, we
can construct the groupoids $G(\Omega, P, T)$ and $G(\partial\Omega,
P, T)$. (Recall that the graph of an action is always
  $(r,d)$-proper.) 
On the other hand, if the action is directed, we can construct
the groupoid $G(X,P,T)$. It is then natural to compare these groupoids
when the action is directed and $P$ is quasi-lattice ordered. An
important case is when $P$ is both directed and quasi-lattice ordered,
which means that $P$ is lattice ordered.

\begin{prop} Let $(X,P,T)$ be a directed semigroup action, where $P$
  is a quasi-lattice ordered subsemigroup of a group $Q$. Then there
  is a $P$-equivariant homeomorphism of $X$ onto $\partial\Omega$
  which implements a groupoid isomorphism of $G(X,P,T)$ and
  $G(\partial\Omega, P, T)$.
\end{prop}

\begin{proof} Given $x\in X$, $A(x)=x\Lambda=\{(x,n)\in X\times P:
  x\in U(n)\}$ is a closed subset of $\Lambda$. It is hereditary. It
  is directed because the action is directed. This defines a map $J:
  X\to\Omega$ sending $x$ to $J(x)=x\Lambda$ which is injective. The
  map $J$ is continuous. Indeed, let $K=L*F$, where $L$ is a compact
  subset of $X$ and $F$ a finite subset of $P$ and $U_i=V_i*\{p_i\}$,
  $i=1,\ldots, n$, where $V_i$ is an an open subset of $X$ and $p_i\in
  P$. Then, $J(x)\cap K=\emptyset$ and $J(x)\cap U_i\not=\emptyset$
  for all $i=1,\ldots, n$ if and only if
  \begin{equation*}
    x\in (L^c\cup r(d^{-1}(F^c))\cap V_1\cap\ldots\cap V_n,
  \end{equation*}
  which is open. The inverse map is continuous since it is the
  restriction to $J(X)$ of the range map $r:\Omega\to X$.

  We claim that $J(x)$ is a boundary path. Because the action is
  directed, every subset $E$ of $\Lambda$ is exhaustive. Let us show
  that $(x,m)$, where $x\in U(m)$, is extendable in $J(x)$. Let $E$ be
  a subset of $\Lambda$ such that $x\cdot m=s(x,m)\in r(E)$. There
  exists $n\in P$ such that $(x\cdot m,n)\in E$. Then $x\in U(mn)$ and
  $(x,mn)=(x,m)(x\cdot m,n)$.

  We also claim that every boundary path $A$ is of the form $J(x)$,
  where $x=r(A)$. We have $A\subset J(x)$ by definition. Conversely,
  let $(x,n)\in J(x)$. Let $E=L*\sset{n}$ where $L$ is a compact
  neighborhood of $x$. Since $A$ is extendable in $A$ at
  $\lambda=(x,e)\in A$, there exists $\mu\in E$ such that
  $\lambda\mu\in A$. Necessarily $\mu=(x,n)$ and $(x,n)=(x,e)(x,n)$
  belongs to $A$.

  Let us show that the map $J:X\to\Omega$ is $P$-equivariant: for all
  $(x,m)\in X*P$, $(J(x),m)\in \Omega*P$ and $J(x\cdot m)=J(x)\ndot
  m$. Let $(x,m)\in X*P$. Since $(x,m)=(x,m)(x\cdot m,e)$ belongs to
  $J(x)$, $J(x)\ndot m$ is non-empty. In fact, we have $J(x)\ndot
  m=J(x\cdot m)$: given that $x\in U(m)$, $(y,n)\in J(x)\ndot m$ means
  exactly that $y=xm$ and $x\in U(mn)$ while $(y,n)\in J(xm)$ means
  that $y=xm$ and $xm\in U(n)$.

  Thus, the semigroup actions $(X,P,T)$ and $(\partial\Omega, P, T)$
  are isomorphic. One deduces that the semi-direct product groupoids
  $G(X,P,T)$ and $G(\partial\Omega, P, T)$ are isomorphic.
\end{proof}

\begin{remark}\label{rem-last}
  (a) Consider a locally compact semigroup action $(X,P,T)$. We are
  able to construct the semi-direct product groupoid $G(X,P,T)$, hence
  the C*-algebra $C^*(G(X,P,T))$, when the action is directed. If the
  action is not directed but $P$ is quasi-lattice ordered, we can
  introduce the topological higher rank graph $\Lambda=X*P$ and
  consider instead the semigroup action $(\partial\Omega,P,T)$, the
  groupoid $G(\partial\Omega,P,T)$ and the Cuntz-Krieger C*-algebra
  $C^*(\partial\Lambda)$. Both constructions agree when they are
  possible.
  
  (b) It is instructive to specialize the situation (a) to the case
  where $X$ is reduced to a point. It turns out that this leads us to
  Wiener-Hopf (also called Toeplitz) C*-algebras of semigroups. We
  have seen earlier that, in this case, the semidirect product
  $G(X,P,T)$ is $PP^{-1}$. If $P$ is an Ore semigroup, then $PP^{-1}$
  is a group and the corresponding C*-algebra is its group
  C*-algebra. If $P$ is not an Ore semigroup, we cannot even define a
  C*-algebra. However, if $P$ is quasi-lattice ordered, we can perform
  the construction of (a) which introduces the higher rank graph
  $\Lambda=P$. Nica shows in \cite{nic:jot92} that the corresponding
  path space $\Omega$ is the spectrum of a canonical diagonal
  C*-subalgebra of the Wiener-Hopf algebra ${\mathcal W}(Q,P)$. In
  fact, he defines in Section 9 of the preprint version of
  \cite{nic:jot92} an \'etale groupoid ${\mathcal G}$,
  which he calls the Wiener-Hopf groupoid and which is exactly (up to
  an obvious isomorphism) our groupoid $G(\Omega, P, T)$. This
  Wiener-Hopf groupoid is 
  only briefly mentioned in the subsection 1.5 of the published
  version  \cite{nic:jot92}. The reduced 
  C*-algebra of this groupoid is the Wiener-Hopf C*-algebra ${\mathcal
    W}(Q,P)$ while the non-degenerate representations of its full
  C*-algebra are exactly the Nica-covariant representations of
  $(Q,P)$. He also shows that, when $P\subset Q$ is the free semigroup
  $SF_n\subset F_n$, the reduction to the boundary of the Wiener-Hopf
  groupoid, which we have denoted earlier by $G(\partial\Omega, P,
  T)$, is the Cuntz groupoid ${\mathcal O}_n$ of
  \cite{ren:groupoid}. Therefore its reduced and its full C*-algebras
  coincide and are isomorphic to the Cuntz algebra $O_n$. One can also
  note that, in this particular case (b) of the general theory of
  higher rank graphs, the boundary path space $\partial\Omega$ is also
  studied by Crisp and Laca \cite{crilac:jfa07}. Their
  definition agrees with ours. We refer the reader to the work on
  semigroup C*-algebras by X.~Li \cites{li:jfa12,li:am13} and by
  Sundar \cite{sun:xx15} for recent developments.

  (c) Let us justify the assertion made earlier that the map
  $F:\Lambda\to {\mathcal C}(\Lambda)$ is not necessarily a
  homeomorphism onto its image. Let $(X,T)$ be a singly generated
  dynamical system in the sense of \cite{ren:otm00}: $X$ is a locally
  compact space and $T$ is a local homeomorphism from an open subset
  dom($T$) of $X$ onto another open subset $\operatorname{ran}(T$) of
  $X$. For $n\in\N$, let $U(n)=\operatorname{dom}(T^n)$. We view $T$
  as an action of $\N$ on $X$, with $X*\N=\{(x,n)\in X\times\N: x\in
  U(n)\}$and $xn=T^nx$.  The associated $\N$-graph is
  $\Lambda=X*\N$. Let $x_\alpha$ be a net converging to $x$ in
  $X$. Assume that there exists $n$ such that $x_\alpha\in U(n)$ for
  all $\alpha$ and $x\in U(n-1)\setminus U(n)$. Then $F(x_\alpha,n)$
  converges to $F(x,n-1)$ in $\mathcal{C}(\Lambda)$ but $(x_\alpha,n)$
  does not converge to $(x,n-1)$ in $\Lambda$.
\end{remark}

%%%%%%%%%%%%%%%%%%%%%%%%%%%%%%%%%%%%%%%%%%%%%%
%%
%% 
% \bibliographystyle{amsxport}
% % %
% \bibliography{references-nov01}

\def\noopsort#1{}\def\cprime{$'$} \def\sp{^}

\end{document}